\documentclass[11pt]{extarticle}

\usepackage{geometry}

\usepackage{amsthm}
\usepackage{amsfonts}
\usepackage{graphicx}
\usepackage{epstopdf}
\usepackage{dsfont}
\usepackage{bm}
\usepackage{bbm}
\usepackage{tikz}
\usepackage{pgfplots}
\usetikzlibrary{patterns}
\usepackage{ulem}
\usetikzlibrary{shapes}
\usetikzlibrary{plotmarks}
\usepackage{appendix}
\usepackage{varwidth}
\usepackage{hyperref}

\usepackage{tikz}
\usepackage{pgfplots}
\usetikzlibrary{shapes}
\usepackage{tikz-3dplot}
\usetikzlibrary{decorations.pathreplacing,angles,quotes,patterns}
\usetikzlibrary{decorations.pathmorphing}
\tikzset{snake it/.style={decorate, decoration=snake}}
\usetikzlibrary{positioning}
\usetikzlibrary{patterns}
\usepackage{subfigure}
\usetikzlibrary{patterns,positioning}
\usetikzlibrary{intersections}
\usetikzlibrary{backgrounds,patterns}
\usetikzlibrary{tikzmark,calc,fadings,arrows,shapes,decorations.pathreplacing}
\usetikzlibrary{shapes,positioning,shadings}
\usepgfplotslibrary{fillbetween}

\usepackage{caption}

\usepackage{color}
\usepackage{todonotes}

\usepackage{amsopn}
\RequirePackage{amsmath}
\RequirePackage{fix-cm}

\usepackage{graphicx}
\usepackage{subfigure}
\usepackage{epstopdf}
\usepackage{indentfirst}
\usepackage{color}
\usepackage{bm}
\usepackage{tikz}
\usetikzlibrary{calc}
\usepackage{ifthen}

\usepackage{ulem}

\numberwithin{equation}{section} 

\newcommand\ds{\displaystyle}
\def\O{\Omega}
\def\o{\omega}
\def\G1{{\bf 1}}

\def\S{{\GS}}
\def\d{\delta}

\def\di{\mathrm{dist}}

\def\SO{\mathrm{SO}}
\def\e{\varepsilon}

\def\fv{\frak{v}}
\def\fu{\frak{u}}

\def\C{\mathcal{C}}
\def\id{\text{\bf id}}
\def\sk{\text{Skew}}
\def\sy{\text{Sym}}
\def\Te{{\cal T}_\varepsilon}

\def\fw{\frak w}

\def\Ga{{\bf a}}

\def\Ge{{\bf e}}

\def\ds{\displaystyle}

\def\N{{\mathbb{N}}}

\def\R{{\mathbb{R}}}
\def\Z{{\mathbb{Z}}}

\def\D{{\mathbb{D}}}
\def\R{{\mathbb{R}}}

\def\N{{\mathbb{N}}}
\def\Z{{\mathbb{Z}}}

\def\Uc{{\mathcal{U}}}

\def\Vc{{\mathcal{V}}}

\def\Yc{{\mathcal{Y}}}

\def\Mc{{\mathcal{M}}}

\def\wh{\widehat }

\def\wt{\widetilde }
\def\X{\times }

\def\Ga{{\bf a}}

\def\Ge{{\bf e}}

\def\Gu{{\bf u}}
\def\Gv{{\bf v}}
\def\Gw{{\bf w}}

\def\GD{{\bf D}}

\def\GI{{\bf I}}
\def\GJ{{\bf J}}

\def\GM{{\bf M}}

\def\GO{{\bf O}}

\def\GQ{{\bf Q}}
\def\GR{{\bf R}}
\def\GS{{\bf S}}

\def\GU{{\bf U}}

\def\GV{{\bf V}}

\definecolor{skyblue}{RGB}{135,206,235}
\definecolor{deepskyblue}{RGB}{0, 191, 255}

\usepackage{amsfonts}
\usepackage{graphicx}
\usepackage{epstopdf}
\usepackage{dsfont}
\usepackage{bm}
\usepackage{bbm}
\usepackage{algorithmic}
\usepackage{tikz}
\usepackage{pgfplots}
\usetikzlibrary{patterns}

\usepackage{nicefrac}
\usepackage{mathtools}

\usepackage{ulem}

\usepackage{url, breakurl}
\usepackage{charter} 
\usepackage{longtable}
\usepackage{float}
\usepackage{supertabular}
\usepackage{graphicx}
\usepackage{color}
\usepackage{subfigure}
\usepackage{overpic}
\usepackage{textcomp}
\usepackage[section]{placeins}

\usepackage{enumitem}

\usepackage{siunitx}
\newtheorem{theorem}{Theorem}
\newtheorem{definition}{Definition}

\newtheorem{assumption}{Assumption}

\newtheorem{remark}{Remark}

\newtheorem{lemma}{Lemma}

\begin{document}


\title{Homogenization, dimension reduction and linearization of thin elastic plate}
\author{Amartya Chakrabortty,  Georges Griso,  Julia Orlik}
\date{}

\maketitle

{\bf Mathematics Subject Classification (2010):} 35B27, 35J86, 35C20, 74K10, 74F10, 76M30, 76M45.
{\let\thefootnote\relax\footnotetext{
	Amartya Chakrabortty, Julia Orlik: SMS, Fraunhofer ITWM, Kaiserslautern 67663, Germany
	
	Emails: amartya.chakrabortty@itwm.fraunhofer.de, julia.orlik@itwm.fraunhofer.de\\
	
	Georges Griso: LJLL, Sorbonne Universit\'e, Paris, France
	
	Email: georges.griso@gmail.com\\
	
	Corresponding author Email: amartya.chakrabortty@itwm.fraunhofer.de}}

\begin{abstract}
           This paper investigates the homogenization, dimension reduction, and linearization of a composite plate subjected to external loading within the framework of non-linear elasticity problem. The total elastic energy of the problem is of order $\sim h^2\varepsilon^{2a+3}$, where $a\geq1$. The paper is divided into two parts: The first part presents the simultaneous homogenization, dimension reduction and linearization ($(\e,h)\to(0,0)$) of a composite plate without any coupling assumption of $\e$ and $h$. The second part consists of the rigorous derivation of linearized elasticity as a limit of non-linear elasticity with small deformation and external loading conditions. The results obtained demonstrate that the limit energy remains unchanged when first linearization ($h\to 0$) is performed, followed by simultaneous homogenization dimension reduction ($\e\to0$) and when both limits approach zero simultaneously, i.e. $(\e,h)\to (0,0)$.
The exact form of the limit energy(s) is obtained through the decomposition of plate deformations and plate displacements. By using the $\Gamma$-convergence technique, the existence of a unique solution for the limit linearized homogenized energy problem is demonstrated. These results are then extended to certain periodic perforated plates. 
\end{abstract}

		\section{Introduction}		
		The periodic thin composite plate is characterized by two parameters: thickness $\d\in (0,1)$ and period $\e\in (0,1)$. Two asymptotic analysis are considered: dimension reduction as $\d\to0$ and homogenization as $\e\to0$. When both parameters tends zero simultaneously, this process is referred to as simultaneous homogenization and dimension reduction $(\e,\d)\to (0,0)$. \\
Given that $(\e,\d)\in (0,1)^2$, it follows that the limit 
\begin{equation}
	\lim_{(\e,\d)\to (0,0)} {\d\over \e}=\kappa\in [0,+\infty]
\end{equation} 
exists for at least a subsequence. In the present work, it is assumed that this limit exists for the entire sequence, with limit $\kappa\in (0,+\infty)$, i.e. $\d$ and $\e$ are of same order. For simplicity of notation, $\d=\e$ is chosen. Thus, as $\e\to0$, simultaneous homogenization and dimension reduction (SHD) is achieved.

\medskip
This paper presents a theoretical asymptotic analysis of a composite plate with external loading that is periodic with respect to $\e$ in two in-plane directions and has a thickness of $2\e$ in the framework of non-linear elasticity. The paper is divided into two parts. The first part deals with Simultaneous Homogenization Dimension Reduction and Linearization (SHDL). The total elastic energy of the system is scaled to be of order $h^2\e^{2a+3}$, where $a \geq 1$. In the second part, it is shown that linear elasticity arises as a limit of non-linear elasticity in the realm of small deformation and external loading.  The study demonstrates that the same limit is obtained when linearization is performed first, followed by SHD as well as when both asymptotic analyses are performed simultaneously. The primary focus of this work is on the critical case of $a=1$, commonly referred to as the von-K\'arm\'an regime.

\medskip
The domain of discussion is a periodic composite plate defined as $\O_\e=\o\X(-\e,\e)$ with $\o\subset  \R^2$ be a bounded domain with Lipschitz boundary. The total energy with external loading is given by
$$		\GJ_{\e,h}(v)\doteq \int_{\O_{\e}}\wh{W}_{\e}(x,\nabla v)-\int_{\O_{\e}}f_{h,\e}\cdot (v-\id) dx,$$
where $ f_{h,\e}$ is the external loading. 
Then 
the minimization problem reads as
$$m_{\e,h}=\inf_{v\in\GV_\e}\GJ_{\e,h}(v),$$
where $\GV_\e$ contains the set of admissible deformations. The external forces are scaled in such a way that the total elastic energy is of order $h^2\e^{2a+3}$, i.e.
$$\int_{\O_\e}\wh{W}_\e(x,\nabla v)\,dx\sim \GO(h^2\e^{2a+3})\quad \text{with $a\geq 1$}.$$

 The main result of the first part of the paper is given in Theorem \ref{MainConR} (SHDL), which states that
	\begin{equation*}
		\lim_{(\e,h)\to (0,0)}{m_{\e,h}\over h^2\e^{2a+3}}=\min_{(\Vc,\wh \fv)}\GJ(\Vc,\wh\fv)=\GJ(\Uc,\wh\fu),
	\end{equation*}
		where 
	$$\GJ(\Uc,\wh\fu)=\int_{\o\X\Yc}\GQ(y,E^{Lin}(\Uc)+e_y(\wh\fu))\,dx'dy-|\Yc|\int_{\o} f\cdot\Uc\;dx'.$$
	In the second part of the paper, first the linearization result is presented in Theorem \ref{TH4},  which states
	\begin{equation*}
		\lim_{h\to 0}{m_{\e,h}\over h^2\e^{2a+3}}={m_\e\over \e^{2a+3}}={\GJ_\e^{Lin}(\Gu_\e)\over\e^{2a+3}},
	\end{equation*}
	where
	\begin{equation*}
		m_\e=\inf_{\Gv_\e\in\GU_\e}\GJ^{Lin}_\e(\Gv_\e)\quad
		\text{with}\quad
		\GJ^{Lin}_\e(\Gu_\e)=\int_{\O_\e}\GQ\left({x\over \e},e(\Gu_\e)\right)dx-\int_{\O_\e}f_\e\cdot\Gu_\e dx,
	\end{equation*}
	followed by SHD, as shown in Theorem \ref{TH03}, where it is established that
	\begin{equation*}
		\lim_{\e\to0}\lim_{h\to0}{m_{\e,h} \over h^2\e^{2a+3}} = \min_{(\Vc,\wh\fv)\in \D} \GJ(\Vc,\wh\fv)=\GJ(\Uc,\wh\fu).
	\end{equation*}
	Finally, the main commutavity result is presented in the Theorem \ref{MainRe} which states that
	$$\lim_{(\e,h)\to(0,0)}{m_{\e,h}\over h^2 \e^{2a+3}}=\lim_{\e\to0}\lim_{h\to0}{m_{\e,h}\over h^2\e^{2a+3}}\GJ(\Uc,\wh\fu).$$
	
The linearized energy, derived from the re-scaled nonlinear energy, is discussed in \cite{masoL} for the stationary case. In that papers, the authors demonstrate the $\Gamma$-convergence limit of finite elasticity as linearized elasticity.  In this paper, a form of $\Gamma$-convergence is used to linearize the nonlinear elastic energy with small external loading and to prove the existence of a minimizer for the limit homogenized energy problem. The total elastic energy analyzed in this paper is a heterogeneous, an-isotropic, frame indifferent, hyperelastic, non-polyconvex, and non-elliptic (see \eqref{34}--\eqref{Eq62}). Generally, a minimizer for such total energy does not exist on the set of admissible deformations $\GV_\e$ (see \eqref{SAD01}). No additional assumptions are made on the existence of solutions for the linearization and simultaneous homogenization dimension reduction. A classical example of such energy in the isotropic case is the St.\ Venant-Kirchhoff material (see Remark \ref{Rem02}). For more details on homogenization using $\Gamma$-convergence, please refer to \cite{masoG}, \cite{ABGamma}, \cite{Friesecke06}, and \cite{FJMKarman}. The commutativity between homogenization and linearization is demonstrated in the periodic setting in \cite{neuL} and in the non-periodic setting in \cite{armstrong} and \cite{LinS1}. However, it should be noted that homogenization does not commute with dimension reduction. For further details, see \cite{NonComHD} and \cite{CDG}.
\begin{figure}[h]
	\centering
	\begin{tikzpicture}[scale=1.3, baseline={(0,-2.5)}]
		\node[draw] (myPoint) at (0,0) {\Large${m_{\e,h}\over h^2 \e^{2a+3}}$};
		\node[draw] (myPoint) at (3.4,0) {\Large$m_L$};
		\node[draw] (myPoint) at (1.95,-1.2) {\Large${m_\e\over \e^{2a+3}}$};
		\draw[->, red, thick] (0.75,0) -- (2.9,0);
		\draw[->, red, thick] (0.63,-0.37) -- (1.40,-0.87);
		\draw[->, red, thick] (2.45,-0.87) -- (3,-0.3);
		\node (myPoint) at (1.7,0.3) {\large$(\e,h)\to(0,0)$};
		\node[below, rotate=-33] at (1,-0.75) {\large$h\to0$};
		\node[below, rotate=49] at (2.75,-0.70) {\large$\e\to0$};
	\end{tikzpicture}		
	\caption{Relation between Linearization, SHD and SHDL for $a\geq 1$.}
\end{figure}
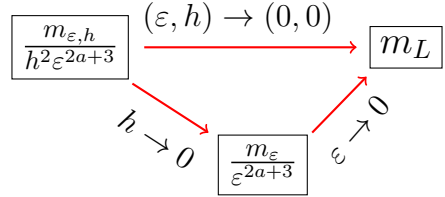

The re-scaling unfolding operator is the tool used for SHD. This operator is a variation of the periodic unfolding operator, which was initially introduced in \cite{CDG2} and further developed in \cite{CDG1}. The unfolding operator is specifically suitable for periodic homogenization problems that depend on the parameter $\varepsilon$. For a more comprehensive understanding of homogenization techniques, please refer to \cite{tartar1}, \cite{asymptotic1}, \cite{Homogenization1}, \cite{twoscale1}, and \cite{twoscale2}. For homogenization in perforated domain using unfolding operator, see \cite{grisoHoles1}--\cite{grisoHoles2}.  Additionally, for further literature on dimension reduction, consult \cite{multi1} and \cite{friesecke02}. The approach presented in this paper, which involves homogenization and dimension reduction using the re-scaling unfolding operator, has similarities to the approach presented in \cite{GJMContact}, \cite{Amartya}, \cite{hauck2}, and \cite{GOW}. The key tool used to derive the exact limit energy is the plate deformation decomposition from \cite{Shell1} and the plate displacement decomposition from \cite{GKL}. Further decomposition results can be found in \cite{BGRod} and \cite{BGRod1}.

\medskip
Previous works have explored the topic of asymptotic analysis (homogenization and dimension reduction) in elasticity in the von-K\'arm\'an regime, references such as \cite{NV}, \cite{Velcic1}, \cite{GOW2} and \cite{Amartya} can be consulted. For general references on the theory of elasticity, \cite{elasticity}, \cite{CiarletShell1}, and \cite{elasticityO} are recommended. For non-linear plate theory, the seminal works mentioned in \cite{friesecke02}, \cite{Friesecke06}, \cite{FJMKarman} and \cite{Shell1} are highly recommended.

\medskip
The paper is structured as follows: Initially, general notations are introduced in Section \ref{N-1}. Following that, the composite periodic domain and the non-linear elasticity problem are presented in Section \ref{Sec03}. In the subsequent section \ref{Sec04}, preliminary results such as plate deformation and displacement decomposition are provided, and the periodic unfolding operator is recalled. Moving to Section \ref{Sec06}, assumptions regarding the external loadings are made, and the re-scaled minimization sequence is established. Sections \ref{Sec08} and \ref{Sec05} present the asymptotic analysis of simultaneous homogenization, dimension reduction, and linearization (SHDL) and simultaneous homogenization and dimension reduction (SHD), respectively. Finally, the results are extended to a certain type of perforated plates in Section \ref{Last}.
\section{Notations} \label{N-1}
Throughout the paper, the following notation will be used:
\begin{itemize}
	\item  $\e,h\in \R^+$ are small parameters and $\kappa>0$ is a real positive number;
	\item $\Yc:=Y\times (-1, 1 )$ be the reference cell of the structure, where $Y\subset \R^2$ bounded domain;
	\item $x=(x',x_3)=(x_1,x_2,x_3)\in \R^3$, $\ds \partial_i\doteq\frac{\partial}{\partial x_i}$ denote the partial derivative w.r. to $x_i$, $i\in \{1,2,3\}$;
	\item $\id$ denote the identity map from $\R^3$ to $\R^3$, $\id_2$ denotes the identity map from $\R^2$ to $\R^3$ and $\GI_3$ denote the identity matrix of order $3$;
	\item $\GM_3$ is the space of $3\times 3$  real matrices, $\S_3$ is the space space of real $3\X3$ symmetric matrices and
	$$\GM^+_3=\{F\in\GM_3: \text{det}(F)>0\};$$
	\item the mapping $v:\O_\e\to\R^3$ denotes the deformation arising in response to forces and pre-strain, it is related to the displacement $u$ by $u=v-\text{\bf id}$. The gradient of deformation and displacement is denoted by $\nabla v$ and $\nabla u$ respectively, with $\nabla v, \nabla u\in \GM_3$ for every $x\in \O_\e$;
	\item  $|\cdot|_F$ denote the Frobenius norm on $\R^{3\X3}$;
	\item  for every displacement $u\in H^1(\O_\e)^3$, the linearized strain tensor is given by
	$$e(u)={1\over 2}\big(\nabla u+(\nabla u)^T\big)\quad \text{and}\quad e_{ij}(u)={1\over 2}\left({\partial u_i\over \partial x_j}+{\partial u_j\over \partial x_i}\right)$$
	where $i,j\in \{1,2,3\}$;
	\item for every $v\in H^1(\O_\e)^3$, $\GD(v)$ denote
	$$\GD(v)\doteq\|\di(\nabla v,\SO(3))\|_{L^2(\O_\e)}.$$
	\item for every deformation $v\in H^1(\O_\e)^3$ the Green St. Venant's strain tensor is given by
	$${1\over 2}\left((\nabla v)^T(\nabla v)-\GI_3\right)={1\over 2}\left((\nabla u)^T+\nabla u+(\nabla u)^T(\nabla u)\right),$$
	where $u= v-\text{\bf id}$ is corresponding displacement;
	\item 	for any $\psi\in H^1(\Yc)^3$, (resp. $L^2\big(\o; H^1(\Yc)\big)^3$) we denote $e_y(\psi)$ the $3\X 3$ symmetric matrix whose entries are
	$$e_{y,ij}(\psi)={1\over 2}\left({\partial \psi_i\over \partial y_j}+{\partial \psi_j\over \partial y_i}\right);$$
	\item for every $F\in \GM_3$, denote by 
\begin{equation}\label{SSE}
\sy(F)={1\over 2}\left(F+F^T\right),\quad\sk(F)={1\over 2}\left(F-F^T\right),\quad\text{and}\quad {\bf E}(F)={1\over 2}\left(F^TF-\GI_3\right);
\end{equation}
	\item $(\alpha,\beta)\in \{1,2\}^2$ and $(i,j,k,l)\in \{1,2,3\}^4$ (if not specified).
		\item By convention, in estimates it is simply written $L^2(\O_\e)$ instead of $L^2(\O_\e)^3$ or $L^2(\O_\e)^{3\times 3}$. The complete space is written when the weak or strong convergences are given.
\end{itemize}
	In this paper, the Einstein convention of summation over repeated indices has been used. Whenever, it is not explicitly mentioned $C$ represents a generic constant independent of $\e,h$. These are some of the general notations used in this paper, notation which are not defined here are defined in the main content of the paper.

\section{Domain description and the non-linear elasticity problem}\label{Sec03}

\subsection{Description of the structure and natural assumption} 
Let $\o$ be a bounded domain in $\R^2$ with Lipschitz boundary. The reference domain is given by
$$\O_\e=\o\X(-\e,\e).$$

Let $Y\subset\R^2$ be a bounded domain in $\R^2$ such that, it has paving property with respect to the additive subgroup $\Z^2$ of $\R^2$. In periodic setting $x'\in\R^2$ can be decomposed a.e. as
$$ x'=\left(\e\left[{x'\over \e}\right]+\e\left\{{x'\over\e}\right\}\right),\quad \text{where}\;\; \left[{x'\over \e}\right]\in\Z^2,\;\;\left\{{x'\over\e}\right\}\in Y.$$
Set
$$\Xi_\e=\left\{\xi\in\Z^2|\e\xi+\e Y\subset \o\right\},\quad \wh{\o}_\e=\text{interior}\left\{\bigcup_{\xi\in\Xi_\e}(\e\xi+\e\overline{Y})\right\},\quad \Lambda_\e=(\o\setminus\overline{\wh{\o}_\e}),$$
where the set $\Lambda_\e$ contains the parts of the cells intersecting the boundary $\partial \o$. We also set $\wh \O_\e=\wh \o_\e\X (-\e,\e)$.

It is assumed that the composite plate $\O_\e$ is clamped on the part of boundary $\Gamma_\e=\gamma\X(-\e,\e)$, where $\gamma\subset \partial\o$ with non-zero area. Then, the set of admissible deformations is denoted by
\begin{equation}\label{SAD01}
	\begin{aligned}
		\GV_{\e}\doteq\Big\{v\in H^1(\O_{\e})^3\;|\;v=\id\quad \hbox{on }\; \Gamma_\e\Big\}.
	\end{aligned}
\end{equation}
 \begin{remark}\label{Re04}
	For every $v\in\GV_\e$, the corresponding re-scaled displacement $\Gu$ given by $v=\id+h\Gu$ satisfies the following
	$$\Gu=0,\quad \text{on}\quad\Gamma_\e.$$
	Let us define the set of admissible re-scaled displacements by
	\begin{equation}
		\GU_\e=\left\{\Gu\in H^1(\O_\e)^3\;|\;\Gu=0\quad\text{on}\;\;\Gamma_\e\right\}.
	\end{equation}
	For a $h>0$, let $\Gu\in\GU_\e$ then observe that $v=\id+h\Gu\in \GV_\e$.
\end{remark}  
\subsection{The non-linear elasticity problem}
In this section, the elasticity problem is introduced and it begins by establishing the initial setup.
It is assumed $a_{ijkl}\in L^{\infty}(\Yc)$, $i,j,k,l\in\{1,2,3\}$ is the Hooke's coefficient and it satisfy both the symmetry condition
\begin{equation}\label{S33}
	a_{ijkl}(y)=a_{jikl}(y)=a_{klji}(y)\quad \text{for a.e.}\quad y\in\Yc,
\end{equation}
and the coercivity condition $(K_0>0)$
\begin{equation}\label{CoeCon}
	a_{ijkl}(y)S_{ij}S_{kl}\geq K_0S_{ij}S_{ij}=K_0|S|_F^2\quad\text{for a.e.}\;\;y\in\Yc\quad\text{and for all}\;\; S\in\GS_3.
\end{equation}
The coefficients $a_{ijkl}^{\e}$ of the Hooke's tensor on the periodic composite plate for $x\in\O_{\e}$ are given by
\begin{equation}
	a_{ijkl}^{\e}=a_{ijkl}\left(\left\{{x'\over \e}\right\},{x_3\over \e}\right),\quad \text{for a.e.}\;\;x\in\O_{\e}.
\end{equation}
For a given applied force $f_{\e}$, the total energy $\GJ_{\e,h}$ reads as 
\begin{equation}\label{34}
	\GJ_{\e,h}(v)\doteq \int_{\O_{\e}}\wh{W}_{\e}(x,\nabla v)\,dx-\int_{\O_{\e}}f_{\e,h}\cdot (v-\id) dx,\quad \text{for all}\;\;v\in\GV_\e,
\end{equation} 
with
the local density energy $\wh W_{\e}\, :\ \,  \O_\e\times \GM_3 \longrightarrow [0,+\infty]$ given by
\begin{equation}\label{Eq62}
	\wh{W}_{\e}(F)=\widehat{W}_{\e}(x , F)=\left\{
	\begin{aligned}
		&\begin{aligned}
			&  \GQ\left(\frac{x}{\e} , \text{\bf E}(F)\right) &&\hbox{if}\;\;  \hbox{det}(F)> 0,\\
			& +\infty  &&\hbox{if}\;\;  \hbox{det}(F)\leq 0,
		\end{aligned}\quad\hbox{for a.e. }  x \in \O_\e.
	\end{aligned}\right.
\end{equation}
The quadratic form $\GQ$ is defined by
$$\GQ\left({x\over\e},S\right)= a^\e_{ijkl}\left({x\over \e}\right)S_{ij}S_{kl} \quad \hbox{for a.e. $x\in\O_\e$ and for all $S\in \GS_3$}.$$
The inequality \eqref{CoeCon} yields
\begin{equation}\label{eq64}
	K_1|F^{T}F-\GI_3|^2_F\leq  \GQ\Big({x\over \e} , {1\over 2}\big (F^{T}F -\GI_3\big)\Big)\qquad \hbox{for a.e. } x\in \O_\e.
\end{equation} 
It is recalled (see e.g. \cite{Friesecke06})
\begin{equation}\label{SymCon}
	\di(F, \SO(3))\leq |F^TF-\GI_3|_F \quad \text{for all}\;\;F \in \GM_3\quad\hbox{such that}\;\; \det(F)>0.
\end{equation}  
For a given applied force $f_{\e,h}$, we set
	\begin{equation}\label{MainPro}
	 m_{\e,h}=\inf_{v\in\GV_\e}\GJ_{\e,h}(v), 
	\end{equation}
	In Section \ref{Sec06} (see \eqref{Req65}) with applied forces given by \eqref{FA01}, it is shown that there exists $c_0>0$ independent of $\e$ and $h$ such that (for $a\geq 1$)
	\begin{equation*}
	-c_0\leq {m_{\e,h}\over h^2\e^{2a+3}} \leq 0.
\end{equation*}
	{\bf It should be noted that the question of whether this infimum is reached remains open.}

\begin{remark}\label{Rem02}
	St. Venant-Kirchhoff's  material is a classical example of local elastic energy satisfying the above assumptions, and which is given by
	$$
	\widehat{W}(F)=\left\{
	\begin{aligned}
		&{\lambda\over 8}\big(tr(F^TF-\GI_3) \big)^2+{\mu\over 4}\left|F^T F-\GI_3\right|^2_F\qquad&&\hbox{if}\quad \det(F)>0\\
		&+\infty &&\hbox{if}\quad\det(F)\le 0.
	\end{aligned}
	\right.
	$$ 
\end{remark}
For the purpose of analysis, it is assumed that the applied forces $f_{\e,h}$ belong  to $ L^2(\O_\e)^3$, where  $f_{\e,h}$ is scaled from  $f\in L^2(\o)^3$ using $\e$ and $h$ which is given by (for $a\geq 1$)
\begin{equation}\label{FA01}
		f_{\e,h}(x)=h f_\e(x)=h\big(\e^{a+1}f_\alpha(x')\Ge_\alpha+\e^{a+2} f_3(x')\Ge_3\big),\quad x\in\O_\e,\quad f\in [L^2(\o)]^3.
\end{equation}
\section{Preliminary results}\label{Sec04}
\subsection{Decomposition of plate deformations}\label{SS41}
Let $v\in\GV_\e$ be a plate deformation. Recalling a result for decomposition of plate deformations from \cite{Shell1}.
\begin{lemma}[Theorem 3.3, \cite{Shell1}]\label{Lem01}
	There exists a constant $C(\o)$ which only depends on $\o$ such that any $v\in H^1(\O_\e)^3$ be a plate deformation then
	\begin{equation}\label{EQ51}
		v = V_{e} + \overline{v} \;\; \text{ a.e. in } \;\; \O_\e.
	\end{equation}
	The quantity   $V_{e} \in H^1{(\O_\e)}^3$ is called elementary plate deformation and is defined by 
	\begin{equation}\label{Eqelem}
		V_e(x)= \mathcal{V}(x_1,x_2) +x_3{\GR}(x_1,x_2) \Ge _{3}  \quad \hbox{for a.e. } x\in \O_\e
	\end{equation}  with $\Vc\in H^1(\o)^3$ and $\GR\in H^1(\o)^{3\X3}$. $\overline{v} \in H^1{(\O_\e)}^3$ is a residual displacement.\\
	The above fields satisfy the following estimates
	\begin{equation}\label{E0}
		\begin{aligned}
			\begin{aligned}
				&\|\overline{v}\|_{L^2(\O_\e)}+\e\|\nabla \overline{v}\|_{L^2(\O_\e)}+ \e\|\nabla \Gv-{\GR}\|_{L^2(\O_\e)}\le C\e\GD(v),\\
				&\big\| \partial_\alpha\GR\big\|_{L^2(\o)}\le {C\over \e^{3/2}}\GD(v),\\
				&\big\|\partial_\alpha\mathcal{V}-{\GR}\Ge_\alpha\big\|_{L^2(\o)}\leq {C \over \e^{1/2}} \GD(v).
			\end{aligned}
		\end{aligned}
	\end{equation}  The constant does not depend on $\e$.
\end{lemma}
Moreover, if  $v\in\GV_\e$ then, we can construct the fields $V_e$ and $\overline{v}$ such that the following boundary conditions hold (see  \cite{Shell1}):
\begin{equation}
	\Vc=\id\quad\text{and}\quad\GR=\GI_3\quad\text{a.e. on}\;\;\gamma,\quad\text{with}\quad\overline{v}=0,\quad\text{a.e. on}\;\;\Gamma_\e.                                                           
	\end{equation}
Now, from \eqref{E0} one has
\begin{lemma}\label{FinEst1}
	The terms $\GR$ and $\Vc$ of the decomposition \eqref{Eqelem} satisfy
	\begin{equation}\label{FinEst} 
		\begin{aligned}
			&\|\GR-\GI_3\|_{H^1(\o)}+\|\Vc-\id\|_{H^1(\o)} \leq {C\over \e^{3/2}} \GD(v).
		\end{aligned}
	\end{equation} The constants do not depend on $\e$.
\end{lemma}	

\begin{proof}
	Estimate  \eqref{E0}$_4$, the boundary conditions and Poincar\'e inequality give \eqref{FinEst}$_1$. Then, from \eqref{FinEst}$_1$ and \eqref{E0}$_4$ we obtain
	$$\|\partial_\alpha\Vc-\Ge_\alpha\|_{L^2(\o)}\leq {C\over \e^{3/2}}\GD(v).
	$$
	The above and  the boundary conditions on $\Vc$ one obtains \eqref{FinEst}$_{2}$.
\end{proof}	
From, the above inequalities, the non-linear Korn-type inequality is obtained for the composite plate $\O_\e$.
\begin{lemma}[Theorem 4.2 and Theorem 4.3 \cite{Shell1}]\label{KornB1}
	The corresponding displacement $u=v-\id$ satisfies
	\begin{equation}\label{EQ448} 
		\|u\|_{H^1(\O_\e)}\leq{C\over\e}\GD(v).
	\end{equation}
	The strain tensor satisfies
	\begin{equation} \label{Eq II.3.10}
		\|e(u)\|_{L^2(\O_\e)} \leq  C_0 \left(\GD(v) + {1 \over \e^{5/2}}\GD(v)^2\right).
	\end{equation}
	The constants do not depend on $\e$.
\end{lemma}
\subsection{Decomposition of plate displacements}
In this subsection, every element $u\in H^1(\O_\e)^3$ is decomposed as the sum of a Kirchhoff-Love displacement and a residual term.\\
Below recalling a result proven in \cite{GKL}.
\begin{lemma}[Theorem 6.1 in \cite{GKL}]\label{DeComKL1} Every displacement belonging to $H^1(\O_\e)^3$ can be decomposed into the sum of a Kirchhoff-Love displacement and a residual term
	\begin{equation}\label{PD1}
		u= U_{KL} + \fu,\qquad \text{  a.e. in }\;\; x\in\O_\e,
	\end{equation} 
	where $U_{KL}$ is a Kirchhoff-Love displacement given by 
	$$
	U_{KL}  = \left(\begin{aligned}
		&\Uc_1 (x') - x_3 {\partial \Uc_3 \over \partial x_1}(x')\\
		&\Uc_2 (x') - x_3 {\partial \Uc_3 \over \partial x_2}(x')\\
		&\hspace{0.9cm}\Uc_3 (x')
	\end{aligned}\right)
	\;\; \text{for a.e.}\;\; x=(x', x_3) \in \O_\e.
	$$
	We have
	\begin{equation}
		\Uc_m=\Uc_1 \Ge_1 +\Uc_2 \Ge_2 \in H^1(\o)^2,\quad \Uc_3\in H^2(\o),\qquad \fu\in H^1(\O_\e)^3
	\end{equation}  and the following estimates
	\begin{equation}\label{PD2}
		\begin{aligned}
			\|e_{\alpha \beta}(\Uc_{m})\|_{L^2(\o)} &\leq {C \over \e^{1/2}} \|e(u)\|_{L^2(\O_\e)},\\
			\|D^2(\Uc_3)\|_{L^2(\o)} &\leq {C \over \e^{3/2} }\|e(u)\|_{L^2(\O_\e)},\\
			\|\fu\|_{L^2(\O_\e)} +\e\|\nabla\fu\|_{L^2(\O_\e)} &\leq C\e \|e(u)\|_{L^2(\O_\e)}.
		\end{aligned}
	\end{equation}
	The constants do not depend on $\e$ and $h$, depends only on $\o$.
\end{lemma}
\begin{remark}
	If $u=0$ on $\Gamma_\e$, one has
	\begin{equation}
		\Uc=0,\quad \nabla\Uc_3=0,\quad\text{a.e. on}\quad \gamma,\quad \text{and}\quad \fu=0\quad\text{a.e. on}\quad \Gamma_\e.
	\end{equation}
\end{remark}
The above Lemma and Remark imply,
\begin{lemma}[Korn type inequalities] The following estimates hold for $u\in H^1(\O_\e)^3$ with $u=0$ on $\Gamma_\e$
	\begin{equation}\label{KornPUD1}
		\begin{aligned}
			\|u_1\|_{L^2(\O_\e)}+\|u_2\|_{L^2(\O_\e)}+\e\|u_3\|_{L^2(\O_\e)}&\leq C_1\|e(u)\|_{L^2(\O_\e)},\\
			\sum_{\alpha,\beta=1}^2\left\|{\partial u_\beta\over \partial x_\alpha}\right\|_{L^2(\O_\e)}+\left\|{\partial u_3\over \partial x_3}\right\|_{L^2(\O_\e)}&\leq C\|e(u)\|_{L^2(\O_\e)},\\
			\sum_{\alpha=1}^2\left(\left\|{\partial u_3\over \partial x_\alpha}\right\|_{L^2(\O_\e)}+\left\|{\partial u_\alpha\over \partial x_3}\right\|_{L^2(\O_\e)}\right)&\leq {C\over\e}\|e(u)\|_{L^2(\O_\e)}.
		\end{aligned}
	\end{equation}
	The constants do not depend on $\e$ and $h$.
\end{lemma}
\subsection{The unfolding operators}
The convergences are done via the re-scaling unfolding operator $\Pi_\e$ for homogenization and dimension reduction in $\O_\e$ and the unfolding operator $\Te$ for homogenization in $\o$.
Below recalling the definition of the periodic unfolding re-scaling operator and the unfolding operator for functions defined in  $\O_\e$ and $\o$, respectively. For the properties of the operators see \cite{CDG}.
\begin{definition}
	For every measurable function $\psi$ on $\O_\e$ the unfolding rescaling operator $\Pi_\e$ is defined by 
	\begin{equation}\label{RUO1}
		\Pi_\e(\psi) (x',y) \doteq \left\{\begin{aligned}
			&\psi \left(\e\left[{x'\over \e}\right] + \e y\right) \qquad &&\hbox{for a.e. }\; (x',y)\in  \wh{\o}_\e\times \Yc,\\
			&0\qquad&&\hbox{for a.e. }\; (x',y)\in  \Lambda_\e\times \Yc.
		\end{aligned}\right.
	\end{equation} 
	For every measurable function $\phi$ on $\o$ the unfolding operator $\Te$ is defined by 
	$$ \Te(\phi) (x',y') \doteq \left\{\begin{aligned}
		&\psi \left(\e\left[{x'\over \e}\right] + \e y'\right) \qquad &&\hbox{for a.e. }\; (x',y')\in  \wh{\o}_\e\times Y,\\
		&0\qquad&&\hbox{for a.e. }\; (x',y')\in  \Lambda_\e\times Y.
	\end{aligned}\right.$$
\end{definition} 
The re-scaling unfolding operator $\Pi_\e$ is a  continuous linear operator  from $L^2(\O_\e)$ into $L^2(\o\X \Yc)$ which satisfies
\begin{equation}\label{EQ722}
	\|\Pi_\e(\psi)\|_{L^2(\o\X\Yc)}\leq{C\over \sqrt{\e}}\|\psi\|_{L^2(\O_\e)}\quad \text{ for every}\quad \psi\in L^2(\O_\e),
\end{equation}
and the unfolding operator is also a continuous linear operator from $L^2(\o)$ into $L^2(\o\X Y)$ which satisfies
$$\|\Te(\phi)\|_{L^2(\o\X Y)}\leq C\|\phi\|_{L^2(\o)}\quad\text{for every}\quad \phi\in L^2(\o).$$
The constants $C$ do not depend on $\e,h$.
\begin{remark}
	Here, note that for functions defined in $\o$, one has
	$$\Pi_\e(\phi)(x',y)=\Te(\phi)(x',y')=\phi \left(\e\left[{x'\over \e}\right] + \e y'\right)\quad \text{for every}\quad \phi\in L^2(\o).$$
\end{remark}
Moreover, for every $\psi\in H^1(\O_\e)$ one has (see \cite[Proposition 1.35]{CDG})
$$ \nabla_y\Pi_\e(\psi)(x',y)=\e\Pi_\e(\nabla \psi)(x',y) \quad \hbox{a.e. in}\quad \o\X\Yc.$$
The lemma below is a consequence of Proposition 1.12, Theorem 1.36, Corollary 1.37 and Proposition 1.39 in \cite{CDG}.
\begin{lemma}\label{A11}
	$ $
	\begin{enumerate}
		\item Let $\{w_\e\}_\e$ be a weakly convergent sequence in $L^2(\o)$ with weak limit $w\in L^2(\o)$ then, there exist a $\wh{w}\in L^2(\o\X Y)$ with $\Mc_{Y}(\wh{w})=0$ such that  upto a subsequence
		$$\Te(w_\e)\rightharpoonup w+\wh{w},\quad \text{weakly in $L^2(\o\X Y)$}.$$
		\item Let $\{v_\e\}_\e$ be a weakly in convergent sequence in $H^1(\o)$ with weak limit $v\in H^1(\o)$. Then, there exist a $\wh{v}\in L^2(\o;H^1_{per,0}(Y))$ such that  upto a subsequence
		$$\begin{aligned}
			\Te(w_\e)&\to v,\quad &&\text{strongly in $L^2(\o\X Y)$},\\
			\Te(\nabla v_\e)&\rightharpoonup \nabla v+\nabla_y \wh{v}\quad&&\text{weakly in $L^2(\o\X Y)$}.
		\end{aligned}$$
	\end{enumerate}
\end{lemma}	
\section{Assumption of the right-hand side forces}\label{Sec06}
We recall that the forces $f_{\e,h}$ are given by (for $a\geq 1$ and $(h,\e)\in (0,1]^2$)
\begin{equation}\label{AFB1}
	f_{\e,h}(x)=h f_\e(x)=h\big(\e^{a+1}f_\alpha(x')\Ge_\alpha+\e^{a+2} f_3(x')\Ge_3\big),\quad x\in\O_\e,\quad f\in [L^2(\o)]^3.
\end{equation}
\begin{assumption}\label{AssS01}
	We assume that the applied forces satisfy the following:
	\begin{equation}\label{Eq74}
		\|f\|_{L^2(\o)}\leq {K_1\over 2C_0 C_1},
	\end{equation}
	where $K_1$, $C_0$ and $C_1$ comes from the estimates \eqref{eq64}, \eqref{Eq II.3.10} and \eqref{KornPUD1}$_1$ respectively.
\end{assumption}
\begin{lemma}\label{L7}
	Let $v\in \GV_{\e}$ be a deformation such that $\GJ_{\e,h}(v)\leq\GJ_{\e,h}(\id)$. Assuming the applied forces \eqref{AFB1} fulfill \eqref{Eq74} and the parameters $h$ and $\e$ satisfy \eqref{AssS02}, then one has  for $a\geq 1$
	\begin{equation}\label{EQ86}
		\begin{aligned}
			\GD(v)+\|e(u)\|_{L^2(\O_\e)}&\leq Ch\e^{a+3/2},\\
			\|(\nabla v)^T(\nabla v)-\GI_3\|_{L^2(\O_\e)}&\leq Ch\e^{a+3/2},
		\end{aligned}
	\end{equation}	
	where $v=\id+u$.
	The constants do not depend on $\e$ and $h$.
\end{lemma}
\begin{proof}
	Observe that $\GJ_{\e,h}(v)\leq \GJ_{\e,h}(\id)=0$, which imply 
	$$\int_{\O_\e}\wh{W}_\e(x,\nabla v)\,dx\leq \int_{\O_\e}f_{\e,h}\cdot(v-\id)\,dx.$$
	Using the inequalities \eqref{eq64}--\eqref{SymCon} and the Lemma  \ref{KornB1}, we obtain the following estimate
	$$
	K_1\GD(v)^2\leq\int_{\O_\e}\wh W_\e(x,\nabla v)\,dx\leq \left|\int_{\O_\e} f_{\e,h}\cdot(v-\id)\, dx\right|.
	$$
	 	The applied forces \eqref{AFB1} and the estimate \eqref{KornPUD1}$_{1}$ imply
	$$\begin{aligned}
	\left|\int_{\O_\e} f_{\e,h}\cdot(v-\id)\, dx\right| =	h\left|\int_{\O_\e} f_\e\cdot u\, dx\right| &\leq h\e^{{a+3/2}} \|f\|_{L^2(\o)}\left(\|u_1\|_{L^2(\O_\e)}+\|u_2\|_{L^2(\O_\e)}+\e\|u_3\|_{L^2(\O_\e)}\right)\\
		&\leq C_1h\e^{a+3/2}\|f\|_{L^2(\o)}\|e(u)\|_{L^2(\O_\e)}.
	\end{aligned}$$
	Then, from the above inequality and \eqref{Eq II.3.10} we obtain 
	\begin{equation*}
		\left|\int_{\O_\e} f_{\e,h}\cdot (v-\id)\, dx\right| \leq C_1h\|f\|_{L^2(\o)}C_0\big(\e^{a+3/2}\GD(v)+\e^{a-1}\GD(v)^2\big).
	\end{equation*}
Thus
	\begin{equation*}
		K_1\GD(v)^2 \leq C_1C_0\|f\|_{L^2(\o)}\big(h\e^{a+3/2}\GD(v)+h\e^{a-1}\GD(v)^2\big).
	\end{equation*}
Since $a\geq 1$, observe that the two positive parameters $(h,\e) \in (0,1]^2$ satisfy the following inequality:
	\begin{equation}\label{AssS02}
		h\e^{a-1}\leq 1.
	\end{equation}
	So, the above inequality along with assumption forces \eqref{Eq74} give
	\begin{equation*}
		{C_1C_0 \over K_1}\|f\|_{L^2(\o)}\leq  {1\over 2},
	\end{equation*} and 
which in-turn give
	$$
	{K_1\over 2}\GD(v)^2\leq C_1C_0h\e^{a+3/2}\GD(v)\leq C_1C_0 \e^{5/2}\GD(v),
	$$
	which imply
	$$\GD(v)\leq Ch\e^{a+3/2}\leq C\e^{5/2}.$$
	Hence, for the corresponding displacement, we obtain
	$$
	\|e(u)\|_{L^2(\O_\e)}\leq C_0\GD(v)+{C_0\over \e^{5/2}}\GD(v)^2\leq Ch\e^{a+3/2}+Ch^2\e^{2a+1/2}\leq Ch\e^{{a+3/2}}\leq C\e^{5/2}.
	$$ The constant does not depend on $h$ and $\e$. So, we have \eqref{EQ86}$_{1,2}$.\\
	The above estimates imply that
	\begin{equation*}
			\left \|\left(\nabla v\right)^T(\nabla v)-\GI_3\right\|^2_{L^2(\O_\e)}
		\leq \int_{\O_\e} \wh{W}_{\e}\big(x,\nabla v\big)\,dx\leq \int_{\O_\e} f_{\e,h}\cdot(v-\id)\,dx \leq Ch^2\e^{2a+3},
	\end{equation*}
	which in turn \eqref{EQ86}$_3$.
	This completes the proof.
\end{proof}
As a consequence, there exists constants $c_0$ strictly positive and independent of $\e,h$ such that 
\begin{equation}
	-c_0h^2\e^{2a+3} \leq \GJ_{\e,h}(v) \leq \GJ_{\e,h}(\id)=0.
\end{equation}
Recalling that $m_{\e,h}=\inf_{v\in \GV_{\e}} \GJ_{\e,h}(v)$ yields 
\begin{equation}\label{Req65}
	-c_0\leq {m_{\e,h}\over h^2\e^{2a+3}} \leq 0.
\end{equation}
 The constant does not depend on $h$ and $\e$.\\[1mm]
{\it \large  So, from now on, we chose $a=1$ the case for $a>1$ is analogous.}\\[1mm]
Two problems of asymptotic analysis are considered:  
\begin{itemize}
	\item Simultaneous homogenization, linearization and dimension reduction (SHDL):
	Asymptotic behavior of the re-scale sequence $\ds\left\{{m_{\e,h} \over h^2\e^{5}}\right\}_{\e,h}$ as $(\e,h)\to (0,0)$ at a time and to characterize its limit as a minimum of the functional.
	\item {\bf First Step} Linearization ($h\to0$ and $\e$ is fixed):
	Asymptotic behavior of the re-scale sequence $\ds\left\{{m_{\e,h} \over h^2 \e^{5}}\right\}_h$ as $h \to 0$ and to characterize its limit as a minimum of the functional.\\
	{\bf Second step} Simultaneous homogenization and dimension reduction (SHD): 	Asymptotic behavior of the re-scale sequence $\ds\left\{{m_{\e} \over \e^{5}}\right\}_\e$ as $\e \to 0$ and to characterize its limit as a minimum of the functional.
\end{itemize}
\section{Simultaneous linearization, homogenization and dimension reduction}\label{Sec08}
In this section, the asymptotic analysis of the sequence $\ds\left\{m_{\e,h}\over h^2\e^{5}\right\}_{\e,h}$ as both $\e,h$ tends to zero together is presented.
\subsection{Asymptotic behavior of a sequence of re-scaled displacements as $(\e,h)\to (0,0)$}
A sequence $\{v_{h,\e}\}_{\e,h}$ of deformations in $\GV_\e$ satisfying
\begin{equation}
	\GD(v_{h,\e})\leq Ch\e^{{5\over 2}},
\end{equation}
is considered.
We set
$$\Gu_{h,\e}={v_{h,\e}-\id\over h}.$$
Then we have using \eqref{EQ86}
\begin{equation}\label{AFLP11}
	\|e(\Gu_{h,\e})\|_{L^2(\O_\e)}\leq C\e^{{5\over2}},
\end{equation}
The asymptotic behavior of the sequence $\left\{(\nabla v_{h,\e})^T(\nabla v_{h,\e})-\GI_3\right\}_{\e,h}$ is provided below.
We have the following identity:
\begin{equation}\label{MainSeq2}
	{\small \begin{aligned}
	{1\over h}(\nabla v_{h,\e})^T(\nabla v_{h,\e})-\GI_3=	{1\over 2h}\Big(\GI_3+h\nabla \Gu_{h,\e}\Big)^T\Big(\GI_3+h\nabla \Gu_{h,\e}\Big)-\GI_3=e(\Gu_{h,\e})+{h\over 2}(\nabla\Gu_{h,\e})^T(\nabla\Gu_{h,\e}).
	\end{aligned}}
\end{equation}
Now, using the Lemma \ref{DeComKL1}, we decompose $\Gu_{h,\e}$ and using the estimates \eqref{PD2}, \eqref{KornPUD1} and \eqref{AFLP11}, one has
\begin{equation}\label{EWAgu1}
	\begin{aligned}
		\|e_{\alpha \beta}(\Uc_{m,h,\e})\|_{L^2(\o)} &\leq C\e^{2},\\
		\|D^2(\Uc_{3,h,\e})\|_{L^2(\o)}&\leq C\e,\\
		\|\fu_{h,\e}\|_{L^2(\O_\e)} +\e\|\nabla\fu_{h,\e}\|_{L^2(\O_\e)} &\leq C\e^{{7\over 2}},\\
		\|\Gu_{1,h,\e}\|_{L^2(\O_\e)}+\|\Gu_{2,h,\e}\|_{L^2(\O_\e)}+\e\|\Gu_{3,h,\e}\|_{L^2(\O_\e)}&\leq C\e^{{5\over2}},\\
		\|\nabla\Gu_{h,\e}\|_{L^2(\O_\e)} &\leq C\e^{{3\over 2}}.
	\end{aligned}
\end{equation}
Then, from the definition of unfolding operator along with the above estimates give
\begin{equation}\label{EWRUO11}
	\begin{aligned}
		\|\Pi_\e(\fu_{h,\e})\|_{L^2(\o\X\Yc)} +\|\nabla_y\Pi_\e(\fu_{h,\e})\|_{L^2(\o\X\Yc)}&\leq C\e^{3},\\
		\|\Pi_\e(\Gu_{1,h,\e})\|_{L^2(\o\X\Yc)}+\|\Pi_\e(\Gu_{2,h,\e})\|_{L^2(\o\X\Yc)}+\e\|\Pi_\e(\Gu_{3,h,\e})\|_{L^2(\o\X\Yc)}&\leq C\e^{2},\\
		\|\Pi_\e(\nabla\Gu_{h,\e})\|_{L^2(\o\X\Yc)}&\leq C\e.
	\end{aligned}
\end{equation}
The constant(s) are independent of $\e,h$ and only depends on $\o$.\\
The following convergences are obtained as $(\e,h)\to (0,0)$ simultaneously.
\begin{lemma}\label{CFuu11}
	There exist a subsequence of $\{h,\e\}$, still denoted by $\{h,\e\}$, and $\Uc_m=\Uc_1\Ge_1+\Uc_2\Ge_2 \in H^1(\o)^2$ and $\Uc_3\in H^2(\o)$ such that
	\begin{equation}\label{CFU11}
		\begin{aligned}
			{1\over \e^{2}}\Uc_{m,h,\e}&\rightharpoonup \Uc_m\quad \text{weakly in}\quad H^1(\o)^2\quad\hbox{and strongly in}\quad L^2(\o)^2,\\
			{1\over \e}\Uc_{3,h,\e}&\rightharpoonup \Uc_3\quad \text{weakly in}\quad H^2(\o)\quad\hbox{and strongly in}\quad H^1(\o). 
		\end{aligned}	
	\end{equation}
	$\Uc_m$, $\Uc_3$ satisfy the following boundary conditions:
	$$ \Uc_m=0\quad \text{a.e. on}\quad \gamma,\quad \Uc_3=\nabla \Uc_3=0\quad\text{a.e. on}\;\; \gamma.$$
		Moreover, there exist $\wh{U}_m\in L^2(\o;H^1_{per,0}(Y))^2$ and $\wh{U}_3\in L^2(\o;H^2_{per,0}(Y))$  such that
	\begin{equation}\label{CFfu11}
		\begin{aligned}
			{1\over \e^{2}}\Te(e_{\alpha\beta}\Uc_{m,h,\e})&\rightharpoonup e_{\alpha\beta}(\Uc_m)+e_{\alpha\beta,y'}(\wh{U}_m),\quad&&\text{weakly in $L^2(\o\X Y)$},\\
			{1\over \e}\Te(D_{\alpha\beta}\Uc_{3,h,\e})&\rightharpoonup D_{\alpha\beta}(\Uc_3)+D_{\alpha\beta,y'}(\wh{U}_3),\quad&&\text{weakly in $L^2(\o\X Y)$},\\
					{1\over \e^{2}}	\Pi_\e\left(e(\Gu_{h,\e})\right) &\rightharpoonup
			E^{Lin}(\Uc)+e_y(\wh{\fu}),\quad &&\text{weakly in}\quad L^2(\o\X \Yc)^{3\X3},
		\end{aligned}
	\end{equation}
		where $E^{Lin}(\Uc)$ denotes the symmetric matrix
	\begin{equation}\label{Eq92}
		E^{Lin}(\Uc)=
		\begin{pmatrix}
			e_{11}(\Uc_m)-y_3D_{11}(\Uc_3)&*&*\\
			e_{12}(\Uc_m)-y_3D_{12}(\Uc_3) &e_{22}(\Uc_m)-y_3D_{22}(\Uc_3) &*\\
			0&0&0
		\end{pmatrix}.
	\end{equation}
	Furthermore, one has
	\begin{equation}\label{C69}
		\begin{aligned}
			{h\over \e^{2}}\left((\nabla\Gu_{h,\e})^T\nabla\Gu_{h,\e}\right)&\rightharpoonup 0\quad &&\text{weakly in}\;\; L^2(\o\X\Yc)^{3\X3}.
		\end{aligned}
	\end{equation}
\end{lemma}	
		\begin{proof}
			The convergences \eqref{CFU11}, \eqref{CFfu11}$_{1,2,3}$ are the immediate consequences of the estimates \eqref{EWAgu1}, \eqref{EWRUO11} and using the properties of the unfolding operator given in Lemma \ref{A11}.\\
				The strain tensor of $\Gu_{h,\e}$ is given by the following $3\X3$ symmetric matrix defined a.e. in $\O_\e$ by 
			\begin{equation}\label{LST01}
				\begin{aligned}
					e(\Gu_{h,\e})&=e(U_{KL,{h,\e}})+e(\fu_{h,\e})\\
					&=\begin{pmatrix}
						\ds e_{11}(\Uc_{m,h,\e})-x_3{\partial^2\Uc_{3,h,\e}\over\partial x_1^2}&*&*\\
						\ds e_{12}(\Uc_{m,h,\e})-x_3{\partial^2\Uc_{3,h,\e}\over\partial x_1\partial x_2}& \ds e_{22}(\Uc_{m,h,\e})-x_3{\partial^2\Uc_{3,h,\e}\over\partial x_2^2}&*\\
						0&0&0
					\end{pmatrix}+e(\fu_{h,\e}).
				\end{aligned}
			\end{equation}
			 From estimate \eqref{EWRUO11}$_{1,2}$, there exits  $\fu \in L^2(\o;H^1_{per}(\Yc))^3$ such that
			$$
			{1\over \e^3}\Pi_\e(\fu_{h,\e}) \rightharpoonup \fu \quad \text{weakly in $L^2(\o;H^1(\Yc))^3$}.
			$$
			Then, the above and the convergences \eqref{CFU11}, \eqref{CFfu11}$_{1,2,3}$ with definition of $e(\Gu_{h,\e})$ we get
			$$\begin{aligned}
				{1\over\e^{2}}\Pi_\e(e(\Gu_{h,\e}))&={1\over\e^{2}}\Pi_\e\left(e(\Uc_{KL,h,\e})\right)+{1\over\e^{3}}e_y\left(\Pi_\e(\fu_{h,\e})\right)\\
				&\rightharpoonup E^{Lin}(\Uc)+E_{y}(\wh{U})+e_y(\fu)\quad \text{weakly in}\quad L^2(\o\X \Yc)^{3\X3},
			\end{aligned}$$
			where
			$$E_{y}(\wh{U})=\begin{pmatrix}
				\ds e_{11,y'}(\wh{U}_{m})-y_3 D_{11,y'}(\wh{U}_{3})&*&*\\
				\ds e_{12,y'}(\wh{U}_{m})-y_3D_{12,y'}(\wh{U}_{3})& \ds e_{22,y'}(\wh{U}_{m})-y_3{D_{22,y'}(\wh{U}_{3})}&*\\
				0&0&0
			\end{pmatrix}.$$
			Setting $\wh{\fu}=\big(\wh{U}_\alpha-y_3\partial_{y_\alpha}\wh{U}_3\big)\Ge_\alpha+\wh{U}_3\Ge_3+\fu$, we get
			$$ E^{Lin}(\Uc)+E_{y}(\wh{U})+e_y(\fu)=E^{Lin}(\Uc)+e_y(\wh{\fu}).$$
			This gives the convergence \eqref{CFfu11}$_3$.\\
Below, we show convergence \eqref{C69}.
			Again using the decomposition \eqref{PD1}, we have $\nabla \Gu_{h,\e}=\nabla \Uc_{KL,h,\e}+\nabla \fu_{h,\e}$, which gives
			$$\nabla \Gu_{h,\e}=
			\begin{pmatrix*}
				\ds \partial_1\Uc_{1,h,\e}-x_3{\partial^2\Uc_{3,h,\e}\over\partial x_1^2} &			\ds \partial_2\Uc_{1,h,\e}-x_3{\partial^2\Uc_{3,h,\e}\over\partial x_1 \partial x_2} &- \partial_1\Uc_{3,h,\e}\\[3mm]
				\ds \partial_1\Uc_{2,h,\e}-x_3{\partial^2\Uc_{3,h,\e}\over\partial x_1\partial x_2} & 	\ds \partial_2\Uc_{2,h,\e}-x_3{\partial^2\Uc_{3,h,\e}\over\partial x_2^2} & -\partial_2\Uc_{3,h,\e}\\[3mm]
				\partial_1\Uc_{3,h,\e}&	 \partial_2\Uc_{3,h,\e}&0
			\end{pmatrix*}
			+ \nabla \fu_{h,\e}.
			$$
			Since the second convergence of \eqref{CFU11} is strong and all other fields vanish due to the convergences \eqref{CFfu11} and \eqref{CFU11}$_1$ together with the estimates \eqref{EWAgu1} and \eqref{EWRUO11}, implies 
			$$
			{1\over \e}\Pi_\e\left(\nabla \Gu_{h,\e}\right)={1\over \e}\Pi_\e(\nabla\Uc_{KL,h,\e})+{1\over\e^{2}}\nabla_y\Pi_\e(\fu_{h,\e})\quad \text{a.e. in}\quad L^2(\o\X \Yc)^{3\X3}.
			$$
			So, we have
			\begin{equation}\label{CFrgu1}
				{h\over \e}\Pi_\e\left(\nabla \Gu_{h,\e}\right) \to 0
				\quad \text{strongly in}\quad L^2(\O\X \Yc)^{3\X3}.
			\end{equation}
			Therefore,
			\begin{equation}\label{C814}
				\begin{aligned}
					{h\over \e^{2}}\Pi_\e\left(\nabla \Gu_{h,\e} (\nabla \Gu_{h,\e})^T\right) &\to 0\quad \text{strongly in}\quad L^1(\o\X \Yc)^{3\X3}.
				\end{aligned}
			\end{equation} 			
			To prove that the convergences \eqref{C814} is also weak in $L^2(\o\X \Yc)^{3\X3}$, for that we show the sequence\\
			$\ds\left\{{h\over \e^{2}}\Pi_\e\left(\nabla \Gu_{h,\e}(\nabla \Gu_{h,\e})^T\right)\right\}_\e$ is bounded in $L^2(\o\X \Yc)^{3\X 3}$. Then, using the estimates \eqref{EQ86}, we obtain
			\begin{equation}\label{Req01}
				\begin{aligned}
					&h\|\nabla \Gu_{h,\e}(\nabla \Gu_{h,\e})^T\|_{L^2(\O_\e)}
					\leq \|h(\nabla \Gu_{h,\e}(\nabla \Gu_{h,\e})^T)+2e(\Gu_{h,\e})\|_{L^2(\O_\e)}+2\|e(\Gu_{h,\e})\|_{L^2(\O_\e)}\\
					&\hskip 33mm ={1\over h}\|(\GI_3+h\nabla \Gu_{h,\e}) (\GI_3+h\nabla\Gu_{h,\e})^T-\GI_3\|_{L^2(\O_\e)}+{2}\|e(\Gu_{h,\e})\|_{L^2(\O_\e)}\leq C\e^{{5\over 2}},
				\end{aligned}
			\end{equation}
			which imply
			$$h\|\Pi_\e(\nabla \Gu_{h,\e}(\nabla \Gu_{h,\e})^T)\|_{L^2(\o\X \Yc)}\leq C\e^{2}.$$
			So the above estimates yields that the convergence \eqref{C814} is also weak in $L^2(\o\X \Yc)^{3\X3}$. 
			This completes the proof.		
		\end{proof}  
		
		As a consequence of the previous Lemma along with the identity \eqref{MainSeq2}, we obtain the  convergence of the non-linear strain tensor.
		\begin{lemma}\label{CFGLST2}
			For the same subsequence as the previous Lemmas, one has when  $(\e,h)\to(0,0)$
			\begin{equation}\label{CFeu1}
				\begin{aligned}
					{1\over 2h\e^{2}} \Pi_\e\left((\nabla v_{h,\e})^T(\nabla v_{h,\e})-\GI_3\right) &\rightharpoonup E^{Lin}(\Uc)+e_y(\wh{\fu})\quad \hbox{weakly in $L^2(\o\X \Yc)^{3\X3}$, }
				\end{aligned}
			\end{equation}
	where $E^{Lin}(\Uc)$ denotes the symmetric matrix given by \eqref{Eq92}.
		\end{lemma}	
Now, we present the main result of this section which is given in Theorem \ref{MainConR}. It characterizes the limit of the linearized re-scaled infimum of the total energy  $\ds {m_{\e,h}\over h^2\e^{5}}={1\over h^2\e^{5}}\inf_{v\in\GV_\e}\GJ_{\e,h}(v)$  as the minimum of the limit energy $\GJ$ over the space $\D$, where 
$$\D=\D_0\X L^2(\o;H^1_{per,0}(\Yc))^3,\quad\text{with}\quad\D_0=\{\Uc\in H^1(\o)^2\X H^2(\o)\;|\;\Uc=\partial_\alpha\Uc_3=0\;\text{a.e. on $\gamma$}\},$$
and
$$\GJ(\Uc, \wh{\fu})=\int_{\o\X\Yc}\GQ\big(y,E^{Lin}(\Uc)+e_y(\wh\fu)\big)\,dydx'- |\Yc|\int_\o f\cdot \Uc\,dx'.$$
So, our limit minization problem is given by
\begin{equation}
	\left\{\begin{aligned}
		&\text{Find $(\Uc,\wh\fu)\in \D$ such that}\\
		&\hskip 15mm\GJ(\Uc,\wh\fu)=m_L=\inf_{(\Vc,\wh \fv)\in \D} \GJ(\Vc,\wh\fv).
	\end{aligned}\right.
\end{equation}
To prove the existence of a unique minimizer of limit energy $\GJ$ in the space $\D$, we recall a norm equivalence result from [Lemma 16, \cite{Amartya}].
\begin{lemma}\label{NEq1}
	Consider the space $\S=\R^3\X\R^3\X H^1_{per,0}(\Yc)^3$ with the semi-norm 
	$$	\|(\eta,\zeta,w)\|_\S = \sqrt{\sum_{i,j=1}^3  \|\wt{\cal E}_{ij}(\eta,\zeta,w)\|^2_{L^2( \Yc)}},
	$$ 
	where for every $(\eta,\zeta,w) \in \S$, we denote $\wt{\cal E}$ the symmetric matrix by
	$$
	\wt{\cal E}(\eta,\zeta,w)=
	\begin{pmatrix}
		\eta_1-y_3\zeta_1+ e_{11,y}( w)  & \eta_3 -y_3 \zeta_3 + e_{12,y}( w) 
		&  e_{13,y}( w)  \\
		* & \eta_2-y_3\zeta_2 + e_{22,y}( w)  &  e_{23,y}( w) \\
		* & *&    e_{33,y}( w) 
	\end{pmatrix}.
	$$ 
	Then, there exists constants $c,C>0$ such that for all $(\eta,\zeta,w) \in \S$ we have
	\begin{equation}\label{ENE1}
		c\left(\|\eta\|^2_2 + \|\zeta\|^2_2 + \|w\|^2_{H^1_{per,0}(\Yc)}\right) \leq  \|(\eta,\zeta,w)\|^2_\S
		\leq C \left(\|\eta\|^2_2 + \|\zeta\|^2_2 + \|w\|^2_{H^1_{per,0}(\Yc)}\right).
	\end{equation} 
\end{lemma}
Hence, using the above inequality \eqref{ENE1} and the Lax-Milgram lemma, we get that the convex energy functional $\GJ$ admits a unique infimum which is in fact a minimum and it is reach for a unique element $(\Uc,\wh \fu)\in \D$, which is the solution of the variational problem
\begin{multline}\label{617}
		{1\over |\Yc|}\int_{\o\X\Yc}\hskip-2mm a_{ijkl}\left(E^{Lin}_{ij}(\Uc)+e_{ij,y}(\wh\fu)\right)\left(E^{Lin}_{kl}(\Vc)+e_{kl,y}(\wh \fv)\right)\,dx'dy=\int_\o f\cdot \Uc\,dx',\quad \forall\, (\Vc,\wh \fv)\in\D. 
\end{multline}
		\begin{theorem}\label{MainConR}
			We have
			\begin{equation}\label{SHDL01}
				m_L=\lim_{(\e,h)\to(0,0)}{m_{\e,h} \over h^2\e^{5}} = \min_{(\Vc,\wh\fv)\in \D} \GJ(\Vc,\fv)=\GJ(\Uc,\wh\fu).
			\end{equation}
		\end{theorem}
		\begin{proof}
			The following proof uses a form of $\Gamma$-convergence.\\[1mm]		
			\noindent{\it {\bf Step 1.}} In this step we show that
			\begin{equation*}
				\liminf_{(\e,h)\to(0,0)} {m_{\e,h}\over h^2\e^{5}}\geq\min_{(\Vc,\fv)\in \D} \GJ(\Vc,\wh\fv)=m_L.
			\end{equation*}	
			Let $\{v_{h,\e}\}_{\e,h} \subset \GV_\e$, be a sequence of deformations, such that it satisfies 
			$$ \lim_{(\e,h)\to(0,0)} {\GJ_{\e,h}(v_{h,\e})\over h^2\e^{5}} = \liminf_{(\e,h)\to (0,0)} {m_{\e,h} \over h^2\e^{5}}.$$
			Without loss of generality, we can assume that the sequence satisfies $\GJ_{\e,h}(v_{h,\e})\leq \GJ_{\e,h}(\id)=0$ and as a consequence of the estimates from the previous section, in particular \eqref{EQ86}, the sequence $\{v_{h,\e}\}_{\e,h}$ satisfies
			\begin{align*}
				\GD(v_{h,\e})+\|e(u_{h,\e})\|_{L^2(\O_\e)}\leq Ch\e^{5/2},\\
				\left\|(\GI_3+h\nabla\Gu_{h,\e})^T(\GI_3+h\nabla\Gu_{h,\e})-\GI_3\right\|_{L^2(\O_\e)}\leq Ch\e^{5/2},
			\end{align*}
			where $v_{h,\e}-\id=u_{h,\e}=h\Gu_{h,\e}$. In particular, we obtain
			\begin{equation}
				\|e(\Gu_{h,\e})\|_{L^2(\O_\e)}\leq C\e^{5/2}.
			\end{equation}
			The constant(s) do not depend on $h$ and $\e$.\\
			Therefore, for any fixed $\e$ and $h$, we are allowed to use the decomposition given in the Lemma \ref{DeComKL1} for the displacement $\Gu_{h,\e}=\ds{v_{h,\e}-\id\over h}$. So, using the decomposition we obtain the estimates \eqref{EWAgu1} and the convergences as in the Lemmas \ref{CFuu11} and \ref{CFGLST2}. Then the assumption on forces \eqref{AFB1} lead to
			\begin{equation*}
				\begin{aligned}
					\lim_{(\e,h)\to(0,0)}{1\over 2h^2\e^{4}}\int_{\o\X\Yc}\Pi_\e(f_{\e,h}\cdot (v_{h,\e}-\id))\,dx'dy=|\Yc|\int_{\o}f\cdot \Uc\,dx'=\int_{\o}F\cdot \Uc\, dx' .
				\end{aligned}
			\end{equation*} 
			We also have 
			$$\begin{aligned}
				{1\over h^2\e^{5}}\int_{\O_\e}\wh{W}_\e\big(x,\nabla v_{h,\e}\big)\,dx\geq {1\over h^2\e^{5}}\int_{\O_\e}\GQ\left({x\over\e},{1\over2}(\nabla v_{h,\e})^T(\nabla v_{h,\e})-\GI_3\right)\,dx.
			\end{aligned}$$
			Observe that
			{\begin{multline*}
					\int_{\O_\e}\GQ\left({x\over\e},{1\over2}(\nabla v_{h,\e})^T(\nabla v_{h,\e})-\GI_3\right)\,dx\\
					=\int_{\wh \O_\e}\GQ\left({x\over\e},{1\over2}(\nabla v_{h,\e})^T(\nabla v_{h,\e})-\GI_3\right)\,dx+\int_{\Lambda_\e\X(-\kappa\e,\kappa\e)}\hskip -4mm\GQ\left({x\over\e},{1\over2}(\nabla v_{h,\e})^T(\nabla v_{h,\e})-\GI_3\right)\,dx.
				\end{multline*}}
			So, the convergence \eqref{CFeu1} along with the fact that $|\Lambda_\e|\to 0$ as $(\e,h)\to(0,0)$ give
			\begin{align*}
				&\lim_{(\e,h)\to(0,0)}{1\over h^2\e^{5}}\int_{ \O_\e}\GQ\Big({x\over\e},{1\over2}(\nabla v_{h,\e})^T(\nabla v_{h,\e})-\GI_3\Big)\,dx\\
				&=\lim_{(\e,h)\to(0,0)}\Big(\int_{\wh\o_\e\X \Yc}\GQ\Big(y,{1\over 2h\e^2}\Pi_\e\Big((\nabla v_{h,\e})^T(\nabla v_{h,\e})-\GI_3\Big)\Big)\,dx'dy\\
				&\hskip 22mm +\int_{\Lambda_\e\X \Yc}\GQ\Big(y,{1\over 2h\e^2}\Pi_\e\Big((\nabla v_{h,\e})^T(\nabla v_{h,\e})-\GI_3\Big)\Big)\,dx'dy\Big)\\
				&=\int_{\o\X \Yc}\GQ\Big(y, E^{Lin}(\Uc)+e_y(\wh \fu)\Big)\,dx'dy.
			\end{align*}
			Hence, as consequence of the above inequality  and using the weak semi-continuity of $\GJ$, we get
			\begin{multline*}
				\liminf_{(\e,h)\to(0,0)}{m_{\e,h}\over h^2\e^5}=\liminf_{(\e,h)\to(0,0)}{\GJ_{\e,h}(v_{h,\e})\over h^2\e^5}\\
				\geq \liminf_{(\e,h)\to(0,0)}\left({1\over h^2\e^5}\int_{\O_\e}\GQ\left({x\over\e},{1\over2}(\nabla v_{h,\e})^T(\nabla v_{h,\e})-\GI_3\right)\,dx-{1\over h^2\e^5}\int_{\O_\e}f_{\e,h}\cdot (v_{h,\e}-\id)\,dx\right)
				=\GJ(\Uc,\wh \fu). 
			\end{multline*}
			This completes the proof of Step 1.\\[1mm]
			\noindent{\it {\bf Step 2.}}
				We recall that there exist $(\Uc,\wh\fu) \in \D$ such that 
				$$ m_L=\min_{(\Vc,\wh\fv)\in \D} {\GJ}(\Vc,\wh\fv) = \GJ(\Uc,\wh\fu).$$
			In this step we show that 
			\begin{equation*}
				\limsup_{(\e,h)\to(0,0)}{m_{\e,h}\over h^2\e^{5}} \leq \GJ(\Uc,\wh{\fu}).
			\end{equation*}
			To do that, we will build a sequence $\{V_{n,h,\e}\}_{n,\e,h}$ of admissible deformations such that 
			$$ \limsup_{(\e,h)\to(0,0)}{m_{\e,h}\over h^2\e^{5}}  \leq \lim_{n\to \infty}\lim_{(\e,h)\to(0,0)} {\GJ_{\e,h}(V_{n,h,\e})\over h^2\e^{5}}=\GJ(\Uc,\wh{\fu}).$$
			We consider a sequence $\{\Uc_n,\wh{\fu}_{n}\}_{n}$ such that
			\begin{itemize}
				\setlength{\itemindent}{5mm}
				\item $\Uc_n\in \D_0\cap \big(\C^1(\overline{\o})^2\X\C^2(\overline{\o})\big)$, satisfying
				\begin{equation}\label{L96}
					\begin{aligned}
						\Uc_{n,\alpha} &\to \Uc_\alpha\qquad \text{strongly in}\quad H^1(\o),\\
						\Uc_{n,3} &\to \Uc_3\qquad \text{strongly in}\quad H^2(\o).
					\end{aligned}
				\end{equation}
				\item  $\wh{\fu}_{n}\in L^2(\o;H^1_{per}(\Yc))^3\cap \C^1_c(\overline{\o}; H^1_{per}(\Yc))^3$ satisfying
				\begin{equation}\label{L97}
					\begin{aligned}
						\wh{\fu}_{n} &\to\fu\qquad \text{strongly in}\quad L^2(\o;H^1_{per}(\Yc))^3.
					\end{aligned}
				\end{equation}
			\end{itemize}  
			We define the following sequence  $\{U_{n,h,\e}\}_{\e}$ of displacements of the whole structure $\O_\e$ as 
			\begin{itemize}
				\setlength{\itemindent}{5mm}
				\item In $\O_\e$ we set
				\begin{equation}\label{L98}
					\begin{aligned}
						U_{n,h,\e,1}&=\e^{2}\left(\Uc_{n,1}(x_1,x_2)-{x_3\over \e}\partial_1\Uc_{n,3}(x_1,x_2)+\e\wh{\fu}_{n,1}\left(x_1,x_2,{x_3\over \e}\right)\right),\\
						U_{n,h,\e,2}&=\e^{2}\left(\Uc_{n,2}(x_1,x_2)-{x_3\over \e}\partial_2\Uc_{n,3}(x_1,x_2)+\e\wh{\fu}_{n,2}\left(x_1,x_2,{x_3\over \e}\right)\right),\\
						U_{n,h,\e,3}&=\e\left(\Uc_{n,3}(x_1,x_2)+\e^2\wh{\fu}_{n,3}\left(x_1,x_2,{x_3\over\e}\right)\right),
					\end{aligned}
				\end{equation}
			\end{itemize}
			By construction, we have the displacements $U_{n,h,\e}$ belong to $\GU_\e$ and satisfy
			$$\begin{aligned}
				\|\nabla U_{n,h,\e}\|_{L^\infty(\O_\e)}&\leq C(n)\e.
			\end{aligned} $$
			which gives the estimate of the displacement gradient. We set
			$$V_{n,h,\e}=\id+hU_{n,h,\e},\quad \text{in}\;\;\O_\e\implies \|\nabla V_{n,h,\e}-\GI_3\|_{L^\infty(\O_\e)}\leq C(n)h\e.$$
			Here the constant $C(n)$ does not depend on $\e$ and $h$, but depends on $n$ such that $C(n)\to +\infty$ for $n\to +\infty$ because  the strong convergence  given in \eqref{L96}--\eqref{L97} are given in $H^1(\o)$, $H^2(\o)$ and $L^2(\o;H^1_{per}(\Yc))^3$. This implies for small enough $\e$ and $h$, we have for a.e. $x\in \O_\e$ we have $|\nabla V_{n,h,\e}(x)-\GI_3|_F<1$. As a consequence we get det$\big(V_{n,h,\e}\big)(x)>0$ for all $n\in \N$ and for a.e. $x\in\O_\e$. This leads to
			\begin{equation}\label{Eq1118}
				\wh{W}_\e(x,\nabla V_{n,h,\e})=\GQ\left({x\over \e},\nabla V_{n,h,\e}\right),\quad \text{for a.e.}\;\;x\in\O_\e,\quad\text{and}\quad m_{\e,h}\leq \GJ_{\e,h}(V_{n,h,\e}).
			\end{equation}
			In the expression \eqref{L98} of the displacement $U_{n,h,\e}$, the explicit dependence with respect to $\e$ and $h$ permits to derive directly the limit of the Green-St.\ Venant's strain tensor as $\e,h$ tends to $0$ ($n$ being fixed) as in Lemma \ref{CFGLST2}, so we obtain
			$$\begin{aligned}
				{1\over 2h\e^{2}} \Pi_\e\left(\big(\nabla V_{n,h,\e}\big)^T (\nabla V_{n,h,\e})-\GI_3\right) &\to E^{Lin}(\Uc_n)+e_y(\wh{\fu}_n),\quad \text{strongly on $L^\infty(\o\X \Yc)^{3\X3}$}.
			\end{aligned}$$
			The above convergences give the convergence of the elastic energy
			$$\begin{aligned}
				\lim_{(\e,h)\to(0,0)}&{1\over h^2\e^{5}}\GJ_{\e,h}(V_{n,h,\e})\\
				& =\lim_{(\e,h)\to(0,0)} {1\over h^2\e^{4}}\int_{\o\X\Yc}\Pi_\e(\wh{W_\e}\big(y, \nabla V_{n,h,\e}\big)\,dx'dy\\
				&=\lim_{(\e,h)\to(0,0)}{1\over h^2\e^{4}}\int_{\o\X \Yc}\Pi_\e\big(\GQ\big(y,\big(\nabla V_{n,h,\e}\big)^T (\nabla V_{n,h,\e})-\GI_3)\big)\,dx'dy\\
				&=\int_{\o\X \Yc}\GQ(y,E^{Lin}(\Uc_n)+e_y(\wh{\fu}_n)\,dx'dy,
			\end{aligned}$$
			and the right-hand side
			$$\lim_{(\e,h)\to(0,0)}{1\over 2h^2\e^{4}}\int_{\o\X\Yc}\Pi_\e(f_{\e,h}\cdot(V_{n,h,\e}-\id))\,dx'dy=  \int_\o F\cdot\Uc_n\,dx'.$$
			Hence, with \eqref{Eq1118} and using the above convergences we obtain
			$$ \limsup_{(\e,h)\to(0,0)}{m_{\e,h}\over h^2\e^{5}}\leq \lim_{(\e,h)\to(0,0)}{1\over h^2\e^{5}}\GJ_{\e,h}(V_{n,h,\e})=\GJ(\Uc_{n},\wh{\fu}_{n}).$$
			Since this holds for every $n\in \N$, when $n$ tends infinity the strong convergences \eqref{L96}--\eqref{L97} yield
			$$\limsup_{(\e,h)\to(0,0)}{m_{\e,h}\over h^2\e^{5}}\leq \lim_{n\to +\infty} \GJ(\Uc_{n},\wh{\fu}_{n})=\GJ(\Uc,\wh{\fu}),$$
			which completes the Step $2$.\\[1mm]			
			\noindent{\it {\bf Step 3.}}
			From Step $1$ and $2$, we have
			$$ \GJ(\Uc,\wh\fu)\leq \liminf_{(h,\e)\to(0,0)}{m_{\e}\over \e^{5}}\leq \limsup_{(h,\e)\to(0,0)}{m_{\e,h}\over h^2\e^{5}}\leq \GJ(\Uc,\wh{\fu}).$$
			This completes the proof.
		\end{proof}
		\subsection{The cell problems}\label{Ssec93}
		To obtain the cell problems, we consider the variational formulation \eqref{617}. From this variational form, we obtain that 
		\begin{equation}\label{Eq1042}
			\begin{aligned}
				& \text{Find $\wh\fu\in L^2(\o;H^1_{per,0}(\Yc))^3$}\; \hbox{satisfies  a.e. in $\o$}\\	
				&\hskip 10mm\begin{aligned}
				\int_{\Yc} a(y)\left(E^{Lin}(\Uc)+e_y(\wh\fu)\right) : e_y(\wh{\fw})\, dy&=0\quad \text{for all}\quad \wh{\fw} \in H^1_{per,0}(\Yc)^3.
				\end{aligned}
			\end{aligned}
		\end{equation}
		So, we have linear problems. Hence there exist a unique solution for the above problems. As a consequence, we obtain 
		\begin{equation}\label{EQ761}
		\int_{\Yc} a(y)e_y(\wh\fu) : e_y(\wh{\fw}) dy=-\int_{\Yc} a(y)\left(E^{Lin}(\Uc)\right) : e_y(\wh{\fw}) dy,
					\quad \forall\; \wh{\fw} \in H^1_{per,0}(\Yc)^3.
		\end{equation}
		This shows that $e_y(\wh\fu)$ can be expressed in terms of the elements of the tensor $E^{Lin}(\Uc)$ and some correctors.\\
			Let us denote by $M^{np}$ the $3\X3$ symmetric matrices with the following coefficients
			\begin{equation}\label{EQMatrix}
				M^{np}_{kl}={1\over 2}\Big(\d_{kn}\d_{lp}+\d_{kp}\d_{ln}\Big),\quad n,p,k,l\in\{1,2,3\},
			\end{equation}
			where $\d_{ij}$ is the Kronecker symbol.
			The cell problems in $\Yc$ are given by
			\begin{equation}\label{Eq125}
				\begin{aligned}
					&\text{Find $\big(\chi^m_{\alpha\beta},\chi^b_{\alpha\beta}\big)\in [H^1_{per,0}(\Yc)^3]^6$ such that}\\
					&\hskip 20mm\begin{aligned}
						&\left.\begin{aligned}
							&\int_{\Yc}a(y)(M^{\alpha\beta}+e_y(\chi^m_{\alpha\beta})) : e_y(w) dy =0,\\
							&\int_{\Yc}a(y)(-y_3M^{\alpha\beta}+e_y(\chi^b_{\alpha\beta})) : e_y(w) dy =0,
						\end{aligned}\right\}\;\;\text{for all}\quad w\in H^1_{per,0}(\Yc)^3.
					\end{aligned}
				\end{aligned}			
			\end{equation}
			The above equations imply for a.e. $(x',y)\in \o\X\Yc$
			\begin{equation}\label{Eq1044}
				\begin{aligned}
					&\wh\fu(x',y)=e_{np}(\Uc_m)(x')\chi^m_{np}(y)+\partial_{np}\Uc_3(x')\chi^b_{np}(y),
				\end{aligned}
			\end{equation}
			where $e_{np}(\Uc_m)$, $\partial_{np}\Uc_3$, $\chi^m_{np}$ and $\chi^b_{np}$ is zero, if $n$ or $p$ is equal to 3.\\
			We set the homogenized coefficients in $\Yc$ as for $x'\in\o$
			\begin{equation}\label{EQ1048}
				\begin{aligned}
					&a^{hom}_{npn'p'} = \frac{1}{|\Yc|} \int_{\Yc} a_{ijkl}(y) \left[M^{np}_{ij} + 
					e_{y,ij}({\chi}_{np}^m)\right]M^{n'p'}_{kl} dy,\\
					&b^{hom}_{npn'p'} = \frac{1}{|\Yc|} \int_{\Yc} a_{ijkl}(y) \left[y_3M^{np}_{ij} + 
					e_{y,ij}({\chi}_{np}^b)\right]M^{n'p'}_{kl} dy,\\
					&c^{hom}_{npn'p'} = \frac{1}{|\Yc|} \int_{\Yc} a_{ijkl}(y) \left[y_3M^{np}_{ij} + 
					e_{y,ij}({\chi}_{np}^b)\right]y_3M^{n'p'}_{kl} dy.
				\end{aligned}
			\end{equation}
			Hence, the homogenized energy is defined by
			\begin{equation}\label{EQ949}
				\GJ_{L}^{hom}(\Uc)=|\Yc|\GJ^{ hom}(\Uc)+|\Yc|\int_{\o}f\cdot \Uc dx',
			\end{equation}
			where
			\begin{equation}\label{EQ950}
				\begin{aligned}
					\GJ^{ hom}(\Uc)&=
					{1\over 2}\int_{\o}\Big(	a^{hom}_{npn'p'}e_{np}(\Uc_m)e_{n'p'}(\Uc_m)+ b^{hom}_{npn'p'}e_{np}(\Uc_m)\partial_{n'p'}\Uc_3\\
					&\hspace{1cm}+c^{hom}_{npn'p'}\partial_{np}\Uc_3\partial_{n'p'}\Uc_3\Big)\,dx'.
				\end{aligned}
			\end{equation}
			The introduction of the homogenized coefficients in the limit energy  gives the convergence result for the homogenized energy. Using the Lemma 11.22 and Remark 11.21 in \cite{CDG}, we have the associated quadratic form of the homogenized energy \eqref{EQ949} is coercive and bounded. Hence, using Lax-Milgram lemma, there exist a unique minimizer of the problem\eqref{EQ949}.\\
			The following is the main result of this section
		\begin{theorem}
			We have for
			\begin{equation}\label{Main01}
				m_L=	\lim_{(\e,h)\to(0,0)}{m_{\e,h} \over h^2\e^{5}} = \min_{(\Vc,\wh{\fv})\in \D} \GJ(\Vc,\wh{\fv})= 	\min_{\Vc\in\D_0} \GJ_{L}^{hom}(\Vc)=\GJ_{L}^{hom}(\Uc),
			\end{equation} 
			where $\Uc$ is the unique minimizer.
		\end{theorem}

\section{Linearization $h\to 0$ ($\e$ being fixed)}\label{Sec05}
In this section, we give general linearization ($h\to0$ and $\e>0$ being fixed) for the total elastic energy with forces. Since, $\e$ is fixed, for simplicity of notation, we only index all terms with $h$.
\begin{lemma}\label{ConML}
	Let $\{v_{h,\e}\}_h$ be sequence of admissible deformations in $\GV_\e$ such that $\GJ_{\e,h}(v_{h,\e})\leq 0$. Then, there exist a subsequence (still denoted by $h$), $\Gu_\e\in\GU_\e$ such that, we have the following convergences:
	\begin{equation}\label{Con1}
		\begin{aligned}
				\Gu_{h,\e}& \rightharpoonup \Gu_\e\quad && \text{weakly in $H^1(\O_\e)^3$},\\
			e(\Gu_{h,\e})&\rightharpoonup e(\Gu_\e)\quad &&\text{weakly in $L^2(\O_\e)^{3\X3}$}
	\end{aligned}
	\end{equation} where $\Gu_{h,\e}=\ds{v_{h,\e}-\id\over h}={u_{h,\e}\over h}$.\\
As a consequence of the above convergence, we also have
\begin{equation}\label{ConLin}
	{1\over 2h}\left(\big(\GI_3+h\nabla\Gu_{h,\e}\big)^T\big(\GI_3+h\nabla\Gu_{h,\e}\big)-\GI_3\right)\rightharpoonup e(\Gu_\e)\;\text{weakly in}\;\; L^2(\O_\e)^{3\X3},
\end{equation}
\end{lemma}
\begin{proof}
	Since $v_{h,\e}$ belongs to $\GV_\e$ and satisfies $\GJ_{\e,h}(v_{h,\e})\leq 0$,  as a consequence of the estimates  \eqref{EQ448}  and \eqref{EQ86}, we have
	$$
	\GD(v_{h,\e})\leq Ch,\qquad \|v_{h,\e}-\id\|_{H^1(\O_\e)}\leq Ch.
	$$	
	The  constant $C$ dependent  on $\e$. We have set
	$h\Gu_{h,\e}=\ds{v_{h,\e}-\id}={u_{h,\e}}$. Thus, we have
	\begin{equation}\label{511}
		\|\Gu_{h,\e}\|_{H^1(\O_\e)}+\|e(\Gu_{h,\e})\|_{L^2(\O_\e)}\leq C.
	\end{equation}
	Thus, there exists a subsequence (still denoted by $h$) and $\Gu_\e\in \GU_\e$ such that the convergences \eqref{Con1} holds.\\
Then, the estimates  \eqref{EQ86}$_3$ and \eqref{511}$_1$ yield
   $$
 \big\|h\big(\nabla \Gu_{h,\e}\big)^T\big(\nabla \Gu_{h,\e}\big) \big\|_{L^1(\O_\e)}\leq h \big\|\nabla \Gu_h\big\|^2_{L^2(\O_\e)}\leq Ch.
$$	
	    This gives
	\begin{equation}\label{E11}
		h\big(\nabla \Gu_{h,\e}\big)^T\big(\nabla \Gu_{h,\e}\big) \to 0\qquad \text{strongly in $L^1(\O_\e)^{3\X3}$}
	\end{equation} 
	Again, using \eqref{511} and proceeding as in \eqref{Req01}, we obtain	
	\begin{align*}
		h\|(\nabla\Gu_{h,\e})^T\nabla\Gu_{h,\e}\|_{L^2(\O_\e)} &\leq C,
	\end{align*}
	which along with the convergence \eqref{E11} gives
	$$	h\big(\nabla \Gu_{h,\e}\big)^T\big(\nabla \Gu_{h,\e}\big) \rightharpoonup 0,\quad \text{weakly in $L^2(\O_\e)^{3\X3}$}.$$
So, using the above convergence along with \eqref{Con1}$_2$ and the identity \eqref{MainSeq2}, we get \eqref{ConLin}.
	\end{proof}
The next theorem is the main linearization result, in which we give the limit of the rescaled sequence $\ds\left\{{m_{\e,h}\over h^2}\right\}_h$ as $h\to0$ ($\e$ is fixed).
\begin{theorem}\label{TH4}
	We have ($\e$ being fixed)
	\begin{equation}
		\lim_{h\to0}{m_{\e,h}\over h^2\e^5}=m_\e,
	\end{equation}
	where
	\begin{equation}\label{Eq1015}
		m_\e=\inf_{\Gu\in\GU_\e}\GJ^{Lin}_\e(\Gu)\quad
		\text{with}\quad
		\GJ^{Lin}_\e(\Gu)=\int_{\O_\e}\GQ\left({x\over \e},e(\Gu)\right)dx-\int_{\O_\e}f_\e\cdot\Gu\, dx.
	\end{equation}
\end{theorem}
\begin{proof}
	The following proof uses a form of $\Gamma$-convergence.\\[1mm]
	We set 
	$$\liminf_{h\to0}{m_{\e,h}\over h^2\e^5} = m_\e.$$
	{\bf Step 1:} In this step, we prove that there exist a unique $\Gu_\e\in \GU_\e$ such that
	\begin{equation}\label{LinUni}
		\GJ_\e^{Lin}(\Gu_\e)=m_\e=\inf_{\Gv\in\GU_\e}\GJ_\e^{Lin}(\Gv),
	\end{equation}
	where
	$$ \GJ_\e^{Lin}(\Gv)=\int_{\O_\e}\GQ\left({x\over \e}, e(\Gv)\right)\,dx-\int_{\O_\e}f_\e\cdot\Gv\,dx.$$
	The variational formulation of the total elastic energy \eqref{Eq1015} for applied forces $f_\e\in L^2(\O_\e)^3$ is given by
	\begin{equation}\label{V01}
		\left\{\begin{aligned}	
		&\hbox{Find $\Gu_\e\in \GU_\e$ such that },\\
			&\hskip 10mm\int_{\O_\e}\sigma^\e_{ij}(\Gu_\e)e_{ij}(w)\,dx=\int_{\O_\e} f_\e\cdot w\,dx,\quad \forall w\in\GU_\e,
		\end{aligned}\right.
	\end{equation}
	where the stress tensor is given by
	$$\sigma^\e_{ij}(w)= a^\e_{ijkl} e_{kl}(w),\quad \text{for all $w\in \GU_\e$}.$$
	We define the bilinear form
	\begin{equation}\label{B01} 
		\Ga_\e(u,w)=\int_{\O_\e}\sigma^\e_{ij}(u)e_{ij}(w)\,dx,\quad \text{for all $u,w$ in $\GU_\e$}.
	\end{equation}
	Observe that since $a_{ijkl}\in L^\infty(\Yc)$ and satisfy the condition \eqref{S33}, we have that bilinear form \eqref{B01} is bounded and symmetric on $\GU_\e\X\GU_\e$. Using inequality \eqref{CoeCon}, we have
	$$K_1\|e(w)\|^2_{L^2(\O_\e)}\leq \Ga_\e(w,w),\quad \forall w\in \GU_\e,\quad \text{for a.e. $x\in\O_\e$}.$$
	Since $w\in \GU_\e$, we have using Korn-type inequalities \eqref{KornPUD1} that there exist a constant $K(\e)>0$ (depending on $\e$) such that
	$$K(\e)\|w\|^2_{H^1(\O_\e)} \leq K_1\|e(w)\|^2_{L^2(\O_\e)}\leq \Ga_\e(w,w),$$
	which imply the bilinear form $\Ga_\e$ satisfies the coerciveness condition.  
	Consequently, by applying the Lax-Milgram lemma with the assumption of the forces that $f_\e\in L^2(\O_\e)^3$, there exist a unique solution $\Gu_\e\in\GU_\e$ for the elasticity problem \eqref{V01} (which is also the unique minimizer of \eqref{Eq1015}). This gives \eqref{LinUni}, which completes Step 1. \\[1mm]	
		{\bf Step 2:} In this step we show that 
	$$\liminf_{h\to0}{m_{\e,h}\over h^2\e^5}\geq m_\e.$$
	Consider a minimizing sequence $\{v_{h,\e}\}_h$ of admissible deformations in $\GV_\e$ satisfying 
	$$\liminf_{h\to 0} {m_{\e,h}\over h^2}=\lim_{h\to 0}{1\over h^2}\GJ_{\e,h}(v_{h})$$ where
	$$\GJ_{\e,h}(v)= \int_{\O_{\e}}\wh{W}_{\e}(x,\nabla v)\,dx- \int_{\O_{\e}} f_{\e,h}\cdot (v-\id)\, dx.$$
	Since each $v_{h,\e}$ is an admissible deformations, so we have $\GJ_{\e,h}(v_{h,\e})\leq \GJ_{\e,h}(\id)=0$.\\	
	Then, we have convergences as in Lemma \ref{ConML}. Similarly, proceeding as in Step 2 of the proof of Theorem \ref{MainConR} give 
	\begin{multline*}
		\liminf_{h\to 0}{m_{\e,h}\over h^2}=	\liminf_{h\to0}{\GJ_{\e,h}(v_{h,\e})\over h^2}\\
		\geq \liminf_{h\to0}\left({1\over h^2}\int_{\O_\e}\GQ\left({x\over\e},{1\over2}\Big(\GI_3+h\nabla \Gu_{h,\e}\Big)^T\Big(\GI_3+h\nabla \Gu_{h,\e}\Big)-\GI_3\right)\,dx-{1\over h^2}\int_{\O_\e}f_{\e,h}\cdot (h\Gu_{h,\e})\,dx\right)\\
		=\int_{\O_\e}\GQ\left({x\over \e},e(\Gu_\e)\right)\,dx-\int_{\O_\e}f_\e\cdot \Gu_\e\,dx=\GJ^{Lin}_\e(\Gu_\e).
	\end{multline*}
	This completes the proof of Step 2.\\[1mm]
		{\bf Step 3:} In this step, we show that 
	$$\limsup_{h\to 0} {m_{\e,h}\over h^2}\leq \GJ_\e^{Lin}(\Gu),\qquad \forall \;\Gu\in\GU_\e.$$
	It is enough to prove the above inequality for all $\Gu\in\GU_\e^*= \GU_\e\cap \C^1(\overline{\O_\e})^3$, since $\GU^*_\e$ is dense in $\GU_\e$.\\
	So, we prove that for any admissible $\Gu\in \GU^*_\e$, there exists a sequence $\{\Gw_{h}\}_h$ of admissible elements in $\GV^*_\e=\GV_\e\cap  \C^1(\overline{\O_\e})^3$  such that
	$${1\over h}(\Gw_{h}-\id)\to \Gu\quad\text{strongly in}\;\; \GU^*_\e$$
	and 
	$$\limsup_{h\to 0}{m_{h}\over h^2}\leq \lim_{h\to0}{\GJ_{\e,h}(\Gw_{h})\over h^2}={\GJ^{Lin}_{\e}(\Gu)}.$$
	Let $\Gu\in \GU_\e\cap \C^1(\overline{\O_\e})^3$ is an re-scaled displacement, then we define the corresponding sequence of deformations by
	$$\Gw_h=\id+h\Gu.$$
	Observe that
	$$\Gw_h\in\GV^*_\e\quad\text{for all $h$}\quad\text{and}\quad\|\nabla \Gw_h-\GI_3\|_{L^\infty(\O_\e)}=\|h\Gu\|_{L^\infty(\O_\e)}\leq Ch.$$
	Here, the constant do not depend on $h$. This implies for small enough $h$, we have $|\nabla \Gw_h-\GI_3|_F<1$ a.e. in $\O_\e$. As a consequence, we get for all $h>0$
	$$\text{det}(\nabla \Gw_h)(x) >0,\quad \text{for a.e.}\;\;x\in \O_\e.$$
	This leads to
	$$\wh{W}_\e(x,\nabla \Gw_h)=\GQ\left({x\over \e},{\bf E}(\nabla \Gw_h )\right),\quad \text{for a.e.}\;\;x\in\O_\e.$$		
	So, the above setting, we have as $h$ tends to $0$ ($\e$ being fixed)
	\begin{equation}\label{C11}
		\begin{aligned}
			{1\over h}(\Gw_{h}-\id)&\to \Gu\quad&&\text{strongly in $L^\infty(\O_\e)^3$},\\
			{1\over 2h}\left((\nabla \Gw_h)^T(\nabla \Gw_h)-\GI_3\right)&\to e(\Gu),\quad &&\text{strongly in $L^\infty(\O_\e)^{3\X3}$},
		\end{aligned}
	\end{equation}
	which gives
	\begin{equation}\label{C12}
		\begin{aligned}
			\lim_{h\to0}{1\over h^2}{\GJ_{\e,h}(\Gw_{h})}&=\lim_{h\to0}\left({1\over h^2}\int_{\O_\e}\GQ\left({x\over \e},{1\over 2}\big((\nabla \Gw_h)^T(\nabla \Gw_h)-\GI_3\big)\right)\,dx-{1\over h}\int_{\O_\e}f_\e\cdot(\Gw_{h}-\id)\,dx\right)\\
			&=\int_{\O_\e}\GQ\left({x\over \e},e(\Gu)\right)dx-\int_{\O_\e}f_\e\cdot\Gu\, dx.
		\end{aligned}
	\end{equation}
 Note that
		$$m_{\e,h} \leq \GJ_{\e,h}(\Gw_{h}).$$
	Then, from the above inequality along with convergence \eqref{C11}-\eqref{C12}, we obtain
	$$\limsup_{h\to0}{m_{\e,h}\over h^2}\leq \lim_{h\to0}{1\over h^2}\GJ_{\e,h}(\Gw_{h})=\GJ^{Lin}_\e(\Gu),\quad \forall \; \Gu\in\GU^*_\e.$$
	This completes the proof of Step 2.\\[1mm]
	Combining steps 1, 2 and 3, we obtain
	$$\GJ^{Lin}_\e(\Gu_\e)={m_\e}\leq \liminf_{h\to0}{m_{\e,h}\over h^2}\leq \limsup_{h\to0}{m_{\e,h}\over h^2}\leq \GJ_\e^{Lin}(\Gu_\e),$$
	which gives the required result.
\end{proof}
Now, we present the main result of the section, where we perform simultaneous homogenization dimension reduction ($\e\to0$) of the linearized energy \eqref{Eq1015}.
\begin{theorem}\label{TH03}
	We have
	\begin{equation*}
		m_L=\lim_{\e\to0}\lim_{h\to 0}{m_{\e,h} \over h^2\e^{5}} = \min_{(\Vc,\wh{\fv})\in \D} \GJ(\Vc,\wh{\fv})=\GJ(\Uc,\wh \fu),
	\end{equation*}
	where 
	$$\D=\D_0\X L^2(\o;H^1_{per,0}(\Yc))^3,\quad\text{with}\quad\D_0=\{\Uc\in H^1(\o)^2\X H^2(\o)\;|\;\Uc=\partial_\alpha\Uc_3=0\;\text{a.e. on $\gamma$}\},$$
	and
	$$\GJ(\Uc,\wh{\fu})=\int_{\o\X\Yc}\GQ\big(y,E^{Lin}(\Uc)+e_y(\wh{\fu})\big)\;dydx'-\int_{\o} F\cdot\Uc\;dx'.$$
\end{theorem}
\begin{proof}
	From the Theorem \ref{TH4}, we have
	$$\lim_{h\to 0} {m_{\e,h}\over h^2}=m_\e=\GJ_\e(\Gu_\e),$$
	which implies (for $\e$ being fixed)
	\begin{equation}\label{Lin02}
		\lim_{h\to 0}{m_{\e,h}\over h^2\e^{5}}={1\over \e^{5}}\lim_{h\to 0}{m_{\e,h}\over h^2}={m_\e\over \e^{5}}={\GJ_\e(\Gu_\e)\over \e^5},
	\end{equation}
	where $m_\e$ is given by \eqref{Eq1015} and $\Gu_\e\in\GU_\e$ is the unique minimizer of \eqref{Eq1015}.\\
	It is enough to show that
	\begin{equation}\label{I01}
		\lim_{\e\to0}{\GJ_\e(\Gu_\e)\over \e^{5}}=\min_{(\Vc,\wh{\fv})\in \D} \GJ(\Vc,\wh{\fv})= \GJ(\Uc,\wh \fu).
	\end{equation}
	From \eqref{Req65}, we have
	$$-c_0\leq {m_{\e,h}\over h^2\e^{5}}\leq 0.$$
	So, from the linearization result Theorem \ref{TH4}, we have
	$$-c_0\leq {1\over \e^{5}}\lim_{h\to0}{m_{\e,h}\over h^2}={m_\e\over \e^{5}}={\GJ_\e^{Lin}(\Gu_\e)\over \e^{5}}\leq 0,$$
	which imply along with \eqref{eq64} and \eqref{AFB1}--\eqref{Eq74}
	$$\begin{aligned}
		K_1\|e(\Gu_\e)\|^2_{L^2(\O_\e)}\leq \int_{\O_\e}\GQ\left({x\over \e},e(\Gu_\e)\right)\,dx\leq \int_{\O_\e} f_\e\cdot\Gu_\e\,dx
		\leq C_0\e^{{5/2}}\|e(\Gu_\e)\|_{L^2(\O_\e)},
	\end{aligned}$$
	which gives
	\begin{equation}\label{814}
		\|e(\Gu_\e)\|_{L^2(\O_\e)}\leq C\e^{5/2}.
	\end{equation}
	Therefore, for any fixed $\e$, we are allowed to use the decomposition given in the Lemma \ref{DeComKL1} for the unique minimizer displacement $\Gu_{\e}$. Then, we have the estimates as in \eqref{EWAgu1}--\eqref{EWRUO11} and convergences as in Lemma \ref{CFuu11}. Then, proceeding as in Theorem 11.19 in \cite{CDG}, we have for the whole sequence $\{\e\}_\e$ and the unique minimizer $\Gu_\e$ of the linear elasticity problem \eqref{LinUni} (also \eqref{V01}), the following convergence holds
		\begin{equation}\label{517}
		\begin{aligned}
			{1\over \e^{2}}\Pi_\e\left(e(\Gu_\e)\right)&\rightharpoonup E^{Lin}(\Uc)+e_y(\wh{\fu})\quad &&\text{weakly in}\quad L^2(\o\X \Yc)^{3\X3},
		\end{aligned}
	\end{equation}  
	 where $(\Uc,\wh\fu)\in\D$ is the unique solution of the re-scaled unfolded problem
$$
{1\over |\Yc|}\int_{\o\X\Yc} a_{ijkl}\left(E^{Lin}_{ij}(\Uc)+e_{ij,y}(\wh\fu)\right)
\left(E^{Lin}_{kl}(\Vc)+e_{kl,y}(\wh \fv)\right)\,dx'dy=\int_\o f\cdot \Uc\,dx',\quad \forall\, (\Vc,\wh \fv)\in\D. 
$$
	 Finally, we obtain \eqref{I01}. 	This completes the proof.
\end{proof}
			
	Proceeding as in Subsection \ref{Ssec93}, we have the final result of this section
	\begin{theorem}
		Under the assumption on the forces \eqref{AFB1}, the problem
		\begin{equation*}
			\min_{\Vc\in\D_0} \GJ_{L}^{hom}(\Vc)
		\end{equation*}	
		admits a unique solution. Moreover, one has
		\begin{equation}\label{EQ1052}
			m_L=	\lim_{\e\to0}{m_{\e} \over \e^{5}} = \min_{(\Vc,\wh{\fv})\in \D} \GJ(\Vc,\wh{\fv})= 	\min_{\Vc\in\D_0} \GJ_{L}^{hom}(\Vc)=\GJ_L^{hom}(\Uc),
		\end{equation} 
		where $\Uc\in\D_0$ is the unique minimizer.
	\end{theorem}
Finally, we present the main result of this paper. Using the convergences \eqref{Lin02}, \eqref{EQ1052}, and \eqref{Main01}, we have
\begin{theorem}\label{MainRe}
We have that the two problems asymptotic analysis converge to the same limit, i.e First linearization followed by simultaneous homogenization dimension reduction give the same resultant limit when linearization, homogenization, and dimension reduction are performed simultaneously, which is given by (for $a\geq 1$)
	\begin{equation*}
		\lim_{(\e,h)\to(0,0)}{m_{\e,h}\over h^2\e^{2a+3}}=\lim_{\e\to0}\lim_{h\to0}{m_{\e,h}\over h^2 \e^{2a+3}}=\GJ_L^{hom}(\Uc)=m_L,
	\end{equation*}
			where $\Uc\in\D_0$ is the unique minimizer.
\end{theorem}
\section{Extension to periodic perforated plates}\label{Last}
In this section, the extension of our results to a general periodic perforated composite plate is undertaken.  We set the following: 
\begin{itemize}
	\item $\Yc^*$ be an  open connected  subset of $\Yc$ with Lipschitz boundary such that interior $\big(\overline{\Yc^*}\cup(\overline{\Yc^*}+\Ge_\alpha)\big)$  is connected and has a Lipschitz boundary for $\alpha=1,2$;
	\item  $\Yc^*$ is  the reference cell of the perforated plate;
	\item $\Xi_\e=\{\xi\in\Z^2\X\{0\}\;|\;\e(\xi+\Yc)\cap\O_\e\neq \emptyset\}$;
	\item $\O_\e^\xi=\e(\xi+\Yc^*)$ and  $\O^*_\e=\ds \bigcup_{\xi\in\Xi_\e}\O_\e^\xi $,
\end{itemize}
\begin{remark}
	Here note the following: $\O_\e^*$ is bounded open connected set (a bounded domain), which  is a $3$-dimensional but only periodic in two directions, in the third direction it is thin with thickness $2\e$.
\end{remark}
Here $\O_\e^*\subset \O_\e$ is the periodic perforated plate such that we have the extension result which is a consequence of the extension result given in see [\cite{GOW2}, Lemma 2.1]. 
\begin{theorem}
	For $v\in H^1(\O^*_\e)^3$ there exist a $\wt{v}\in H^1(\O_\e)^3$ such that
	$$\wt{v}_{|\O_\e^*}=v\quad \text{and}\quad \|\di(\nabla \wt{v},\SO(3))\|_{L^2(\O_\e)}\leq c\|\di(\nabla v, \SO(3))\|_{L^2(\O_\e^*)}.$$
	The constant is independent of $\e$.
\end{theorem}
Hence all our results are true for such perforated plates.

\bibliographystyle{plain}
\bibliography{reference}

\end{document}